\documentclass[10pt,reqno]{article}
\usepackage{amsmath,amssymb,amsthm}
\usepackage{esint}
\usepackage{bm}
\usepackage{subfigure}
\usepackage{graphicx,color}
\usepackage[sort&compress]{natbib}
\usepackage{algorithm}
\usepackage{algpseudocode}
\usepackage{algorithmicx}
\usepackage{hyperref}
\usepackage{hhline}
\usepackage{setspace}
\usepackage{url}
\usepackage{caption}
\hypersetup{colorlinks,citecolor=blue,linkcolor=blue, urlcolor=blue}

\DeclareMathAlphabet{\itbf}{OML}{cmm}{b}{it}

\def\by{{{\itbf y}}}
\def\bx{{{\itbf x}}}
\def\bz{{{\itbf z}}}

\def\etab{{\boldsymbol{\xi}}}

\newcommand{\RR}{\mathbb{R}}

\newcommand{\K}{{\kappa}}

\newcommand{\bu}{\mathbf{u}}
\newcommand{\bv}{\mathbf{v}}

\newcommand{\ds}{\displaystyle}

\newcommand{\Itd}{\mathcal{I}_{\rm TD}}

\newcommand{\sigmae}{{\sigma_\xi}}
\newcommand{\hby}{{\widehat{\mathbf{y}}}}

\newcommand{\hbx}{{\widehat{\bx}}}

\newtheorem{thm}{Theorem}[section]
\newtheorem{cor}[thm]{Corollary}
\newtheorem{lem}[thm]{Lemma}

\newtheorem{defn}[thm]{Definition}

\numberwithin{equation}{section}

\newcommand{\pathfigures}{Figures/}
\graphicspath{{\pathfigures}}

\usepackage{epstopdf}
\begin{document}

\title{Detection of Electromagnetic Inclusions using Topological Sensitivity
\thanks{This research was supported by the Korea Research Fellowship Program funded by the Ministry of Science, ICT and Future Planning through the National Research Foundation of Korea (NRF-2015H1D3A1062400) and by the National Research Foundation of Korea under Grants (NRF-2016R1A2B3008104)  and (NRF-2014R1A2A1A11052491). 
}}
\author{
Abdul Wahab
\thanks{\footnotesize Bio Imaging and Signal Processing Lab., Department of Bio and Brain Engineering, Korea Advanced Institute of Science and Technology, 291 Daehak-ro, Yuseong-gu,
Daejeon 305-701, Korea (wahab@kaist.ac.kr).}\,  
\and
Tasawar Abbas
\thanks{\footnotesize Department of Mathematics and Statistics, FBAS, International Islamic University, 44000, Islamabad, Pakistan (tasawar44@hotmail.com).}
\and
Naveed Ahmed
\thanks{\footnotesize  Weierstrass Institute for Applied Analysis and Stochastics, Leibniz Institute in Forschungsverbund Berlin e. V. (WIAS), Mohrenstr. 39, 10117 Berlin, Germany (naveed.ahmed@wias-berlin.de).}
\and
Qazi Muhammad Zaigham Zia
\thanks{\footnotesize Department of Mathematics, COMSATS Institute of Information Technology, Park Road, Chak Shahzad, 44000, Islamabad, Pakistan (zaighum\_zia@comsats.edu.pk).} 
}
\maketitle

\begin{abstract}
In this article  a topological sensitivity framework for far field detection of a diametrically small electromagnetic inclusion is established. The cases of single and multiple measurements of the electric far field scattering amplitude at a fixed frequency are taken into account. The performance of the algorithm is analyzed theoretically in terms of its resolution and sensitivity for locating an inclusion. The stability of the framework with respect to measurement and medium noises is discussed. Moreover, the quantitative results for signal-to-noise ratio are presented. A few numerical results are presented to illustrate the detection capabilities of the proposed framework with single and multiple measurements. 
\end{abstract}

\noindent {\footnotesize {\bf AMS subject classifications 2000.} Primary, 35L05, 35R30, 74B05; Secondary, 47A52, 65J20}

\noindent {\footnotesize {\bf Key words.} Electromagnetic imaging; Topological derivative;  Localization; Resolution analysis; Stability analysis; Medium noise; Measurement noise.}

\section{Introduction} 

A thriving interest has been shown in topological sensitivity frameworks to procure solutions of assorted inverse problems, especially for detecting small inhomogeneities and cracks embedded in homogeneous media \cite{TDelastic, AGJK, Princeton, Multi, BG, Bellis, CGGM, DG, DGE, GWL, Guzina, HL, MPS, park, park2,  wahab, NCMIP2015, WAHR}. The impetus behind this curiosity is the appositeness and simplicity of these algorithms. Notwithstanding their utility, the use of topological sensitivity based algorithms and their stability with respect to medium and measurement noises is more often heuristic and lacks rigorous quantitative analysis of resolution and signal-to-noise ratio. 

Ammari et al  \cite{AGJK} used asymptotic analysis to substantiate that the detection of small acoustic inclusions embedded in a bounded domain cannot be guaranteed without pre-processing the measurements using a \emph{Calder\'on preconditioner} based on \emph{Neumann-Poincar\'e} operator associated with the domain. Further, the detection of small inclusions nearly touching the boundary  cannot be guaranteed. In \cite{TDelastic}, they established for linear isotropic elastic media that the mode-conversion and non-linear coupling of the wave modes at the boundary degenerate the localization and resolution of the classical imaging framework. Therefore, a weighted topological derivative based sensitivity framework was designed and debated for guaranteed inclusion detection. The stability and robustness of the acoustic and elastic algorithms were also debated in \cite{TDelastic, AGJK}. In particular, it was proved for acoustic inclusion detection using multiple measurements that the topological derivative based imaging functions are more stable than contemporary techniques such as multiple signal classification (MUSIC), back-propagation and Kirchhoff migration. However, its computational complexity is relatively higher than the listed techniques.  The analysis was further extended to the case of electromagnetic inclusion detection in \cite{wahab} wherein the tangential components of the scattered magnetic field at the surface of a bounded domain were used. The results in \cite{wahab} compliment those provided by  Masmoudi et al \cite{MPS}, wherein an adjoint field based approach is used to procure solutions of some inverse problems in non-destructive evaluation.  The far field imaging of acoustic inhomogeneities using topological sensitivity was discussed by Bellis et al  \cite{Bellis} using a factorization of the far field operator.  The case of elastic inclusion detection using topological derivatives in half and full spaces were discussed by Bonnet and Guzina \cite{BGElastic}. The topological sensitivity frameworks have been also developed to ascertain the morphology (shape, size, and material properties) of small inclusions and cracks. The interested readers are referred, for instance, to \cite{Bao, Guzina, bonnet}. Moreover, a level set theory for a variety of location indicator functions was developed in  \cite{Burger} and \cite{DL}. The partial aperture problems and the computational aspects were addressed using multi-frequency approaches, for example, in \cite{Ahn, partial}.  Finally, we precise that the topological sensitivity framework is primarily a \emph{single-short} method and the multi-static configuration is generally used to enhance its performance in terms of stability. We refer, for instance, to  \cite{Li1, Li2, Li3} for other asymptotic frameworks for detection of multiple multiscale electromagnetic inclusions using single-short algorithms. 

The goal in this investigation is to perform a quantitative analysis of a topological sensitivity functional to locate an electromagnetic inclusion of vanishing characteristic size using measurements of the far field scattering amplitude of the electric field. The general case of an electric field satisfying full Maxwell equations is considered and the arguments are constructed using asymptotic analysis by Ammari and Volkov \cite{AV}. The stability with respect to medium and measurement noises is discussed and the estimates for signal-to-noise ratio are furnished. The stability and sensitivity analysis for topological derivative based detection functions for electromagnetic inclusions using far field amplitude is not available in the literature, specially in the context of full Maxwell equations, at least not up to the knowledge of the authors.

First of all, the underlying medium containing an electromagnetic inhomogeneity is probed with an electric incident field and the far field amplitude of the scattered  electric field is recorded at the unit sphere. Then, a trial inhomogeneity is created at a search point in the background medium (without any inclusion) and is probed with the same incident field, rendering another far field amplitude due to the trial inclusion. A discrepancy functional is then constructed and its minimizer is sought. The point relative to which the far field amplitude of the trial inclusion minimizes the discrepancy is a potential candidate for the location of the  true inclusion. Since the optimization problem is difficult to handle directly, the discrepancy is expanded using asymptotic expansion of the far field amplitude  versus scale factor of the inclusion \cite{AV}. The leading order term in the expansion of the discrepancy is coined as its \emph{topological derivative}. By definition, topological derivative synthesizes the sensitivity of the discrepancy relative to the insertion of an inclusion at a search point. A point relative to which the topological derivative attains its most pronounced negative value is the strongest candidate for the true location as the discrepancy admits its most pronounced decrease at that point, given that the contrasts of true and trail inclusions have same signs. The interested readers are referred, for instance, to \cite{Bellis, BG} for further details on sign heuristic.

The inverse problem of inclusion detection is of great importance in diverse domains of science and engineering, especially in bio-medical imaging of cancerous tissues \cite{Cancer}. Moreover, these problems lend importance from the non-destructive testing of material impurities and buried objects such as in-service pipelines and anti-personal land mines \cite{DG}.

The rest of this article is organized in the following manner.  In Section \ref{form}, some notation and key results are collected and the inverse problem of anomaly detection is formulated mathematically. Section \ref{framework} deals with an $L^2-$discrepancy functional and the topological sensitivity based indicator function resulting therefrom. The sensitivity and resolution  analysis of the location indicator function is performed in Section \ref{sensitivity}. The stability with respect to measurement and medium noises is debated respectively  in Section \ref{measNoise} and  Section \ref{medNoise}.  Finally, the important results of the paper are summarized in Section \ref{conc}.

\section{Mathematical formulation}\label{form}

Let $\RR^3$ be loaded with an electromagnetic material (hereafter called the background medium) having permittivity $\epsilon_0\in\RR_+$ and permeability $\mu_0\in\RR_+$. Let the speed of light and wave-number in background medium be defined by $c_0:=1/\sqrt{\epsilon_0\mu_0}$ and $\K:=\omega/c_0$ respectively,  where $\omega\in\RR_+$ is the angular frequency.  

Let $D=\rho B_D+ \bz_D\subset\RR^3$ be a bounded inclusion with a smooth and simply connected boundary $\partial D$,  permittivity $\epsilon_1\in\RR_+$ and permeability $\mu_1\in\RR_+$, where $\bz_D\in\RR^3$ is the vector position of its center, $B_D\subset\RR^3$ is a regular enough bounded domain containing origin and representing the volume of the inclusion, and $\rho\in\RR_+$ is the scale factor. 

Let $\theta\in\mathbb{S}^2=\{\bx\in\RR^3\,|\, \bx\cdot\bx=1\}$ and $\theta^\perp\in\RR^3$ be any vector orthogonal to $\theta$.
Consider an incident time-harmonic electric field $\mathbf{E}_0$ of the form 
\begin{eqnarray}
\mathbf{E}_0(\bx):=\nabla_\bx\times(\theta^\perp e^{i\K\theta^\top\bx}) = i\K\theta\times \theta^\perp e^{i\K\theta^\top\bx}, \qquad \forall\,\bx\in\RR^3,\label{E0def}
\end{eqnarray}
satisfying the equation
\begin{eqnarray}
\nabla\times\nabla\times\mathbf{E}_0(\bx)-\K^2\mathbf{E}_0(\bx)= \mathbf{0}, \qquad \forall\,\bx\in\RR^3.\label{E0}
\end{eqnarray}
Then  the total time harmonic electric field  $\mathbf{E}_\rho$ in $\RR^3$ in the presence of inclusion $D$ is the solution to 
\begin{eqnarray}
\nabla\times\left(\epsilon_0^{-1}\chi_{\RR^3\setminus\overline{D}}+\epsilon_1^{-1}\chi_{D}\right)\nabla\times\mathbf{E}_\rho(\bx)-\omega^2
\left(\mu_0\chi_{\RR^3\setminus\overline{D}}+\mu_1\chi_{D}\right)\mathbf{E}_\rho(\bx)= \mathbf{0}, \quad\forall\, \bx\in\RR^3,\label{Erho}
\end{eqnarray}
subject to the Silver-M\"{u}ller radiation condition
\begin{eqnarray}
\ds\lim_{|\bx|\to\infty}|\bx|\Big[\nabla\times\left(\mathbf{E}_\rho-\mathbf{E}_0\right)\times\widehat{\bx}-i\K\left(\mathbf{E}_\rho-\mathbf{E}_0\right)\Big]=\mathbf{0},\label{ErhoSM}
\end{eqnarray}
where for any vector $\bv\in\RR^3\setminus\{\mathbf{0}\}$,  $\widehat{\bv}:=\bv/|\bv|$ and $\chi_D$ is the characteristic function of $D$.
The existence of a unique weak solution to \eqref{Erho}-\eqref{ErhoSM} is established, for example, in \cite{colton,nedelec}. 

The Silver-M\"{u}ller condition \eqref{ErhoSM} guarantees the existence of a vector field $\mathbf{E}^\infty_\rho\in L^2_T(\mathbb{S}^2)$ coined as \emph{far-field scattering amplitude} or \emph{far-field scattering pattern} of the electric field, which is  analytic on $\mathbb{S}^2$ and is characterized by the far-field expansion 
\begin{equation}
\mathbf{E}_\rho(\bx)-\mathbf{E}_0(\bx)= \frac{e^{i\K|\bx|}}{|\bx|}\mathbf{E}^\infty_\rho\left(\hat\bx;\theta,\theta^\perp, \K\right)+O\left(\frac{1}{|\bx|^2}\right)
\label{Einfty}
\end{equation} 
as $|\bx|\to\infty$ \cite{colton}. 
Here  $L^2_T(\mathbb{S}^2)$ is the space
of all square-integrable tangential vector fields on the unit sphere, that is,
$$
L^2_T(\mathbb{S}^2):=\left\{\hat\bv\in L^2(\mathbb{S}^2;\mathbb{C}^3)\,\,|\,\, \bv(\hbx)\cdot\hbx=0,\quad \forall\,\hbx\in\mathbb{S}^2\right\}.
$$
Since the wave-number is assumed fixed  in this investigation, the dependence of $\mathbf{E}^\infty_\rho$ on $\K$ is suppressed henceforth.  Following problems are dealt with in this article.
\subsubsection*{Inverse problems} 
\begin{enumerate}
\item Let $\theta\in\mathbb{S}^2$ and $\mathbf{E}_0(\bx)$ be the incident field in \eqref{E0} associated with $\theta$ and $\theta^\perp$. Given the measurements $\left\{\mathbf{E}^\infty_\rho(\hat\bx; \theta,\theta^\perp), \,\forall\, \hat\bx\in\mathbb{S}^2\right\}$, find the position $\bz_D$ of the inclusion $D$.

\item Let $\theta_1, \theta_2, \cdots, \theta_n\in\mathbb{S}^2$ for some $n\in\mathbb{N}$ and $\mathbf{E}_0^{j,\ell}$ be the incident fields of the form \eqref{E0} associated with $\theta_j$ and $\theta_j^{\perp,\ell}$ where $\ell\in\{1,2\}$ and $\{\theta_j,\theta_j^{\perp,1},\theta_j^{\perp,2}\}$ forms an orthonormal basis of $\RR^3$. Given the measurements $\{\mathbf{E}^\infty_\rho(\hbx; \theta_j,\theta_j^{\perp,\ell}), \,\forall\, \hat\bx\in\mathbb{S}^2, \,\ell=1,2,\quad j=1,\cdots, n \}$, find the position $\bz_D$ of the inclusion $D$.
\end{enumerate}

Remark that the  measurements are assumed to be available on the entire unit sphere $\mathbb{S}^2$. Since $\mathbf{E}^\infty_\rho$ is analytic on $\mathbb{S}^2$, the knowledge of $\mathbf{E}^\infty_\rho$ on any open subset of it furnishes the measurements on entire $\mathbb{S}^2$ thanks to analytic continuation and consequently  the limited aperture case can be dealt with.

\subsection{Preliminaries and notation}

Before proceeding any further, let us fix some notation and recall a few important results. Let the outgoing fundamental solution of the Helmholtz operator $-(\Delta+\K^2)$ in $\RR^3$ be denoted by $g$. It is well known (see, for instance, \cite{nedelec}) that  
$$
g(\bx,\by):=\frac{e^{i\K|\bx-\by|}}{4\pi|\bx-\by|}, \qquad \bx\neq\by,\quad  \bx,\by\in\RR^3.
$$
By virtue of fundamental solution $g$,  the dyadic fundamental solution $\mathbf{\Gamma}$ is constructed as
$$
\mathbf{\Gamma}(\bx,\by):=-\epsilon_0 \left[\mathbf{I}_3+\frac{1}{\K^2}\nabla_\bx\nabla_\bx^\top\right] g(\bx,\by),\qquad \bx\neq\by,\quad  \bx,\by\in\RR^3,
$$
where $\mathbf{I}_3\in\RR^{3\times 3}$ is the identity matrix. Recall that $\mathbf{\Gamma}(\bx,\by)$ is the solution to 
$$
\ds\nabla_\bx\times \nabla_\bx \times \mathbf{\Gamma}(\bx,\by)-\K^2 \mathbf{\Gamma}(\bx,\by)= -\epsilon_0\delta_\by(\bx)\mathbf{I}_3, \qquad
\bx,\by\in\RR^3 
$$
subject to the \emph{Silver-M\"uller} condition
\begin{eqnarray*}
\ds\lim_{|\bx-\by|\to \infty} |\bx-\by|\left[\nabla_\bx\times\mathbf{\Gamma}(\bx,\by)\times\frac{\bx-\by}{|\bx-\by|}-i\K\mathbf{\Gamma}(\bx,\by)\right]=0.
\end{eqnarray*}
Here $\delta_{a}(\cdot)=\delta_0(\cdot-a)$ is the Dirac mass at $a$ and the operator $\times$ acts column-wise on $\mathbf{\Gamma}$, that is, for any smooth vector field $\bu\in\RR^3\to\RR^3$ 
$$
\left[\bu\times \mathbf{\Gamma}\right]\mathbf{p}=\bu\times \left[\mathbf{\Gamma}\mathbf{p}\right], \quad \forall\,\mathbf{p}\in\RR^3.
$$
It is worthwhile precising that $\mathbf{\Gamma}$ is symmetric and possesses the following reciprocity properties in isotropic  materials (see, for instance, \cite{AILP}),
\begin{eqnarray*}
\mathbf{\Gamma}(\bx,\by)=  \mathbf{\Gamma}(\by,\bx) 
\quad\text{and}\quad
\nabla_\bx\times \mathbf{\Gamma}(\bx,\by)=\left[\nabla_\by\times \mathbf{\Gamma}(\by,\bx)\right]^\top.
\end{eqnarray*}

Let us define the electric \emph{Herglotz} wave operator $\mathcal{H}_E$ by 
\begin{eqnarray}
\mathcal{H}_E[\Phi](\bz):=  \int_{\mathbb{S}^2}\Phi(\hbx) e^{i\K\hbx^\top\bz}ds(\hbx),\qquad\forall\,\Phi\in L^2_T(\mathbb{S}^2),\, \bz\in\RR^3.\label{Herglotz}
\end{eqnarray}
It can be easily verified that for any tangential field $\Phi\in L^2_T(\mathbb{S}^2)$ the Herglotz function $\mathcal{H}_E[\Phi]$ is an entire solution to \eqref{E0}.  Refer to \cite[Chapter 6]{colton} for detailed discussion on Herglotz waves.

Let $\gamma_0\in\{\epsilon_0, \mu_0\}$, $\gamma_1\in\{\epsilon_1,\mu_1\}$ and   $v_i$ be the scalar potential defined as the solution to the transmission problem 
\begin{eqnarray*}
\begin{cases}
\Delta v_i= 0, & \RR^3 \setminus\partial{B_D}
\\
v_i^+-v_i^-=0, & \partial {B_D}
\\
\ds\frac{\gamma_0}{\gamma_1}\left(\frac{\partial v_i}{\partial\nu}\right)^+ -\left(\frac{\partial v_i}{\partial \nu}\right)^-=0, & \partial {B_D}
\\
v_i(\bx)-x_i\to 0, & |\bx|\to \infty. 
\end{cases}
\end{eqnarray*}
Then the polarization tensor $\mathbf{M}_{B_D}({\gamma_0}/{\gamma_1}):= (m_{ij})_{i,j=1}^3$,  associated with inclusion $D$ and depending on the contrast $\gamma_0/\gamma_1$, is defined by 
$$
m_{ij}\left(\frac{\gamma_0}{\gamma_1}\right):=\ds\frac{\gamma_1}{\gamma_0} \int_{{B_D}}\frac{\partial v_i}{\partial x_j}d\bx.
$$

\begin{lem}[See \cite{AKpol}]\label{PT}
The tensor $\mathbf{M}_{B_D}\left({\gamma_0}/{\gamma_1}\right)$ is real symmetric positive definite if $\left({\gamma_0}/{\gamma_1}\right)\in\RR_+$. Moreover, when $D$ is a sphere 
$$
\mathbf{M}_{B_D}\left(\frac{\gamma_0}{\gamma_1}\right)= \dfrac{3\gamma_1}{2\gamma_0+\gamma_1}|{B_D}|\mathbf{I}_3.
$$
\end{lem}

Finally, for $\mathbf{A},\mathbf{B}\in\RR^{3\times 3}$, with components $(a_{ij})$ and $(b_{ij})$ respectively,  the contraction operator (denoted by `$:$') and the \emph{Frobenius norm} (denoted by $\|\cdot\|$) are defined by 
$$
\mathbf{A}:\mathbf{B} := \ds\sum_{i,j=1}^3 a_{ij}b_{ij}
\quad\text{and}\quad
\|\mathbf{A}\|:= \sqrt{\mathbf{A}:\mathbf{A}}.
$$

\section{Location indicator function} \label{framework}

The aim in this article is to recast the aforementioned inverse problems as optimization problems for some $L^2$-discrepancy functionals and debate the performance of the topological sensitivity based location indicator functions resulting therefrom. In order to do so, choose a search point $\bz_S\in\RR^3$ and  create an inclusion $D_S=\delta B_S+\bz_S$ inside the background $\RR^3$, having permittivity $\epsilon_2\in\RR_+$ and permeability $\mu_2\in\RR_+$ with \`a priori known bounded smooth domain $B_S\subset\RR^3$.
Let $\mathbf{E}_\delta$  be the electric field in $\RR^3$ in the presence of inclusion $D_S$ subject to incident field $\mathbf{E}_0$ and satisfying equations analogous to those described in \eqref{Erho}-\eqref{ErhoSM} with $(\epsilon_{1},\mu_{1}, \chi_{D}, D)$ replaced by $(\epsilon_{2},\mu_{2},\chi_{D_S}, D_S)$.  

Let $\mathbf{E}^\infty_\delta$ be the far field scattering amplitude associated with the electric field $\mathbf{E}_\delta$.
Consider the $L^2-$cost functional 
\begin{eqnarray}
\ds\mathcal{J}[\mathbf{E}_0](\bz_S)
:= \frac{1}{2}\int_{\mathbb{S}^2} \left| \mathbf{E}^\infty_{\rho}(\hat\bx;\theta,\theta^\perp)-\mathbf{E}^\infty_{\delta}(\hat\bx;\theta, \theta^\perp)\right|^2 ds(\hbx).\label{Dis}
\end{eqnarray}
Then, by construction, the search point $\bz_S\in\RR^3$ relative to which $\mathbf{E}_\delta^\infty$ minimizes the functional $\mathcal{J}[\mathbf{E}_0]$ is a potential candidate for $\bz_D$. 
In order to address the optimization problem \eqref{Dis},  the \emph{topological derivative} of the functional $\mathcal{J}$  is defined as follows. 
\begin{defn}[Topological derivative] 
For any $\bz_S\in\RR^3$ and incident electric field $\mathbf{E}_0$  the topological derivative of $\mathcal{J}[\mathbf{E}_0]$  at point $\bz_S\in\RR^3$, denoted by ${\partial}_T\mathcal{J}[\mathbf{E}_0](\bz_S)$, is defined by
$$
{\partial}_T\mathcal{J}[\mathbf{E}_0](\bz_S):= \ds\frac{\partial\mathcal{J}[\mathbf{E_0}]}{\partial(\delta)^3}\Big|_{\delta=0}(\bz_S). 
$$
\end{defn}

The topological sensitivity based location indicator function $\Itd[\mathbf{E}_0]$ defined by 
\begin{eqnarray}
\Itd[\mathbf{E}_0](\bz_S):= -{\partial}_T\mathcal{J}[\mathbf{E}_0](\bz_S) 
\label{indicator}
\end{eqnarray}
will be considered in this article for inclusion detection. In the sequel, the shorthand notation 
\begin{align} 
\mathbf{M}_D^\gamma:= \mathbf{M}_{B_D}\left(\frac{\gamma_0}{\gamma_1}\right),
\qquad
\mathbf{M}_S^\gamma:= \mathbf{M}_{B_S}\left(\frac{\gamma_0}{\gamma_2}\right),
\qquad
a_{\ell}^\gamma:=\left(\frac{\gamma_0}{\gamma_\ell}-1\right) 
\label{notabene}
\end{align}
is used, where 
$\ell\in\{1,2\}$ and $\gamma_m=\epsilon_m$ or $\mu_m$ for $m\in\{0,1,2\}$. 

The following asymptotic expansion of the far field scattering amplitude versus scale factor $\rho$ is the key ingredient to   explicitly express  the indicator function  $\Itd[\mathbf{E}_0]$  in terms of the incident electric field $\mathbf{E}_0$, Herglotz field $\mathcal{H}_E[\mathbf{E}_0]$ and the polarization tensors associated with $D_S$. Refer to {\cite[Theorem 1.1]{AV}} for further details on the proof of the asymptotic expansion. 
\begin{thm}\label{Asymp}
Let $\theta\in\mathbb{S}^2$ and $D=\rho B_D+\bz_D$ such that $\rho\K\ll 1$.  Then for all $\bx\in\RR^3$ far from $\bz_D$ with respect to the operating wavelength,
\begin{align}
\nonumber
\ds \mathbf{E}^\infty_\rho(\hbx;\theta,\theta^\perp) 
=-\frac{i\K^3\rho^3}{4\pi}
\Big[
a^\mu_1&\left\{\mathbf{M}_D^\mu\left(\theta\times
\left(\theta\times\theta^\perp\right)\right)\right\}
\times\hbx
\nonumber
\\
&\quad
+a^\epsilon_1
\left(\mathbf{I}_3-\hbx{\hbx}^\top\right)
\mathbf{M}_D^\epsilon\left(\theta\times\theta^\perp\right)
\Big]
e^{i\K(\theta-\hbx)^\top\bz_D} +O(\rho^4).\label{AE1}
\end{align}
The term $O(\rho^4)$ is bounded by $C\rho^4$ uniformly on $\bx$, where  constant $C$ is independent of $\bz_D$.
\end{thm}

By virtue of Theorem \ref{Asymp}, for all $\bx\in\RR^3$ far from $\bz_S$
\begin{align}
\nonumber
\ds \mathbf{E}^\infty_\delta(\hbx;\theta,\theta^\perp) 
=-\frac{i\K^3\delta^3}{4\pi}
\Big[a^\mu_2
&
\left\{\mathbf{M}_S^\mu\left(\theta\times
\left(\theta\times\theta^\perp\right)\right)\right\}
\times\hbx
\nonumber
\\
&
\quad +a_2^\epsilon\left(\mathbf{I}_3-\hbx{\hbx}^\top\right)
\mathbf{M}_S^\epsilon\left(\theta\times\theta^\perp\right)
\Big]
e^{i\K(\theta-\hbx)^\top\bz_S} +O(\delta^4).\label{AE2}
\end{align}
Using asymptotic expansions \eqref{AE1}-\eqref{AE2}, the explicit expression for the location indicator function $\Itd[\mathbf{E}_0]$ can be given in terms of the Herglotz field. Precisely, the following result holds.
\begin{thm}\label{thmTD}
Consider the incident electric field $\mathbf{E}_0$ satisfying \eqref{E0}, contrasts $a_2^\epsilon$ and $a_2^\mu$ defined in \eqref{notabene} and the Herglotz operator $\mathcal{H}_E$ defined in \eqref{Herglotz} and suppose that the condition $\rho\K\ll 1$ is satisfied. Then the location indicator function $\Itd[\mathbf{E}_0]$ at a search point $\bz_S\in\RR^3$ is given by
\begin{align}
\Itd[\mathbf{E}_0](\bz_S)
=
-\ds\frac{1}{4\pi}\Re e\Big\{ 
a^\mu_2\mathbf{M}_S^\mu\nabla\times\mathbf{E}_0(\bz_S)
&
\cdot\nabla\times
\overline{\mathcal{H}_E[\mathbf{E}^\infty_\rho(\cdot,\theta,\theta^\perp)](\bz_S)}
\nonumber
\\
&
+\K^2 a^\epsilon_2\overline{\mathcal{H}_E
[\mathbf{E}^\infty_\rho(\cdot;\theta,\theta^\perp)](\bz_S)}
\cdot \mathbf{M}_S^\epsilon
\mathbf{E}_0(\bz_S)
\Big\},
\label{TD}
\end{align} 
where a superposed bar reflects a complex conjugate. 
\end{thm}

\begin{proof}
In order to derive expression \eqref{TD}  the cost functional $\mathcal{J}[\mathbf{E}_0]$ is expanded as
\begin{align*}
\mathcal{J}[\mathbf{E}_0](\bz_S) 
&= \frac{1}{2}\int_{\mathbb{S}^2}\left|\mathbf{E}^\infty_\rho(\hbx;\theta,\theta^\perp)\right|^2ds(\hbx)
-\Re e\left\{\int_{\mathbb{S}^2}\mathbf{E}^\infty_\delta(\hbx;\theta,\theta^\perp)\cdot\overline{\mathbf{E}^\infty_\rho(\hbx;\theta,\theta^\perp)}ds(\hbx)\right\}
\\
&+\frac{1}{2}\int_{\mathbb{S}^2}\left|\mathbf{E}^\infty_\delta(\hbx;\theta,\theta^\perp)\right|^2ds(\hbx).
\end{align*}
The last term on the right hand side (RHS) is of the order $O(\delta^6)$ since the leading order term in \eqref{AE2} is of the order $O(\delta^3)$. Moreover, on substituting the expression \eqref{AE2} in the second term on the RHS, one can see that 
\begin{align}
\mathcal{J}[\mathbf{E}_0]&(\bz_S) 
-\frac{1}{2}\int_{\mathbb{S}^2}\left|\mathbf{E}^\infty_\rho(\hbx;\theta,\theta^\perp)\right|^2ds(\hbx)
\nonumber
\\
=&
\frac{\delta^3}{4\pi} 
\Re e \Bigg\{i\K^3
a^\mu_2
\int_{\mathbb{S}^2}\Big\{\mathbf{M}_S^\mu
\left(\theta\times(\theta\times\theta^\perp)\right)\Bigg\}
\times\hbx\cdot\overline{\mathbf{E}^\infty_\rho(\hbx;\theta,\theta^\perp)}e^{i\K(\theta-\hbx)^\top\bz_S} ds(\hbx)
\nonumber
\\
&
+i\K^3a^\epsilon_2 
\int_{\mathbb{S}^2}\left(\mathbf{I}_3-\hat\bx{\hat\bx}^\top\right)
\mathbf{M}_S^\epsilon\left(\theta\times\theta^\perp\right)
\cdot\overline{\mathbf{E}^\infty_\rho(\hbx;\theta,\theta^\perp)}e^{i\K(\theta-\hbx)^\top\bz_S} ds(\hbx)
\Big\}
+O(\delta^4) 
\nonumber
\\
=&
\ds\frac{\delta^3}{4\pi}\Re e\Big\{a^\mu_2\Upsilon_1(\bz_S)+a^\epsilon_2\Upsilon_2(\bz_S)\Big\}+O(\delta^4),
\label{TDexp}
\end{align}
where $\Upsilon_1$ and $\Upsilon_2$ are defined by
\begin{align*}
\Upsilon_1(\bz):= &i\K^3\int_{\mathbb{S}^2}
\Big\{\mathbf{M}_S^\mu
\left(\theta\times(\theta\times\theta^\perp)\right)\Big\}
\times\hbx\cdot\overline{\mathbf{E}^\infty_\rho(\hbx;\theta,\theta^\perp)}e^{i\K(\theta-\hbx)^\top\bz} ds(\hbx),
\\
\Upsilon_2(\bz):=& i\K^3\int_{\mathbb{S}^2}
\left(\mathbf{I}_3-\hat\bx{\hat\bx}^\top\right)
\mathbf{M}_S^\epsilon\left(\theta\times\theta^\perp\right)
\cdot\overline{\mathbf{E}^\infty_\rho(\hbx;\theta,\theta^\perp)}e^{i\K(\theta-\hbx)^\top\bz} ds(\hbx).
\end{align*}

Let us analyze $\Upsilon_1$ and $\Upsilon_2$ individually. Recall that 
\begin{align}
\nabla\times\mathbf{E}_0(\bz)&=-\K^2\theta\times(\theta\times\theta^\perp)e^{i\K\theta^\top\bz},\label{CurlE}
\\
\nabla_{\bz}\times\overline{\left(\Phi(\hbx) e^{i\K\hbx^\top\bz}\right)}
&= -i\K\hbx\times\overline{\Phi(\hbx)}e^{-i\K\hbx^\top\bz}, 
\end{align}
for all $\bz\in\RR^3$, $\hbx, \theta\in\mathbb{S}^2$ and 
$\Phi\in L^2_T(\mathbb{S}^2)$. Consequently  
\begin{align}
\nonumber
\Upsilon_1(\bz)
=&
\int_{\mathbb{S}^2}\mathbf{M}_S^\mu
\nabla\times \mathbf{E}_0(\bz)
\cdot\nabla_{\bz}\times
\overline{\left(\mathbf{E}^\infty_\rho(\hbx,\theta,\theta^\perp)e^{i\K\hbx^\top\bz}\right)}ds(\hbx) 
\nonumber
\\
=&\phantom{.}\mathbf{M}_S^\mu 
\nabla\times \mathbf{E}_0(\bz)
\cdot\nabla\times
\overline{\mathcal{H}_E[\mathbf{E}^\infty_\rho(\cdot,\theta,\theta^\perp)](\bz)}. \label{T1}
\end{align}
On the other hand, for all $\bx,\bz,\bu,\bv\in\RR^3$, $\mathbf{A}\in\RR^{3\times 3}$ and $\Phi\in L^2_T(\mathbb{S}^2)$
\begin{align}
\mathbf{A}\bu\cdot\bv 
=&
\phantom{-.}\mathbf{A}^\top\bv\cdot\bu,
\label{relation1}
\\
\left(\mathbf{I}_3-\hbx\hbx^\top\right)\bu
=&
-\hbx\times
(\hbx\times\bu),
\label{relation2}
\\
\nabla_{\bz}\times \nabla_{\bz}\times \left(\overline{\Phi(\hbx)e^{i\K\hbx^\top\bz}}\right)
=&
-\K^2\hbx\times\left(\hbx\times \overline{\Phi(\hbx)}\right)e^{-i\K\hbx^\top\bz}.
\label{relation3}
\end{align}
Therefore, successive use of \eqref{relation1}-\eqref{relation3} renders
\begin{align}
\Upsilon_2(\bz)
=&\phantom{-.} i\K^3\int_{\mathbb{S}^2}
\left(\mathbf{I}_3-\hbx{\hbx}^\top\right)
\overline{\mathbf{E}^\infty_\rho(\hbx;\theta,\theta^\perp)}e^{-i\K\hbx^\top\bz}ds(\hbx)\cdot \mathbf{M}_S^\epsilon
\left(\theta\times\theta^\perp\right)e^{i\K\theta^\top\bz}
\nonumber
\\
=&
-\K^2\int_{\mathbb{S}^2}
\hbx\times\left(\hbx\times
\overline{\mathbf{E}^\infty_\rho(\hbx;\theta,\theta^\perp)}e^{-i\K\hbx^\top\bz}\right)ds(\hbx)\cdot \mathbf{M}_S^\epsilon
\mathbf{E}_0(\bz)
\nonumber
\\
=&
\phantom{-.}\int_{\mathbb{S}^2}
\nabla_{\bz}\times\nabla_{\bz}\times\left(
\overline{\mathbf{E}^\infty_\rho(\hbx;\theta,\theta^\perp)e^{i\K\hbx^\top\bz}}\right)ds(\hbx)\cdot \mathbf{M}_S^\epsilon
\mathbf{E}_0(\bz)
\nonumber
\\
=&
\phantom{-.}\nabla\times\nabla\times
\overline{\mathcal{H}_E[\mathbf{E}^\infty_\rho(\cdot;\theta,\theta^\perp)](\bz)}\cdot \mathbf{M}_S^\epsilon
\mathbf{E}_0(\bz).
\nonumber
\end{align}
Since  $\mathcal{H}_E[\Phi]$ is an entire solution to \eqref{E0} for all  $\Phi\in L^2_T(\mathbb{S}^2)$, it immediately follows that 
\begin{align}
\Upsilon_2(\bz)= 
\K^2\overline{\mathcal{H}_E[\mathbf{E}^\infty_\rho(\cdot;\theta,\theta^\perp)](\bz)}\cdot \mathbf{M}_S^\epsilon
\mathbf{E}_0(\bz).\label{T2}
\end{align}
Finally, the proof  is completed by substituting the expressions \eqref{T1} and \eqref{T2} back in \eqref{TDexp}, taking  the limit $\delta^3\to 0$ and using the definition \eqref{indicator} of $\Itd$.

\end{proof}

\section{Sensitivity and resolution analysis} \label{sensitivity}

This section is dedicated to substantiating  that the indicator function $\Itd$ is felicitous to detect the true location of the inclusion $D$. Precisely, the aim is to establish that the function $\Itd$ admits a sharp peak with focal spot size of the order of half the operating wavelength at search points $\bz_S$ 	approaching $\bz_D$. A location indicator function dealing with multiple measurements will also be constructed using $\Itd$ in this section.

\subsection{Inclusion detection using single measurement}
In order to analyze the performance of the indicator function $\Itd$, two types of inclusions are considered for brevity. It is assumed that the contrast of the inclusion $D$ from the background material is defined either by the permittivity parameter only ($\mu_0=\mu_1$, $\epsilon_0\neq \epsilon_1$) or the permeablility parameter only ($\epsilon_0=\epsilon_1$, $\mu_0\neq \mu_1$). The general case can be dealt with analogously, but the analysis would be more involved. Moreover, additional information about the contrast parameters may be required for guaranteed localization. Refer, for instance, to \cite[Section 4.1.1]{wahab} for detailed discussion. 

\subsubsection{Permittivity contrast}

Let the permeability of $D$ be indifferent from the background medium, that is, $\mu_1=\mu_0$ and the contrast of the inclusion be described by the permittivity $\epsilon_1$  which is different from $\epsilon_0$. Then a trial inclusion is also nucleated with permeability $\mu_2=\mu_0$  whereas its permittivity is assumed to be different from the background. As a result, the first term in explicit form \eqref{TD} of the location indicator function vanishes. Therefore  
\begin{align}
\Itd[\mathbf{E}_0](\bz_S)
=
-\ds\frac{\K^2 a_2^\epsilon}{4\pi}\Re e\left\{
\overline{\mathcal{H}_E
[\mathbf{E}^\infty_\rho(\cdot;\theta,\theta^\perp)](\bz_S)}
\cdot \mathbf{M}_S^\epsilon
\mathbf{E}_0(\bz_S)\right\},
\label{TDeps}
\end{align}  
where the asymptotic expansion of the Herglotz field $\mathcal{H}_E[\mathbf{E}_\rho^\infty]$ versus scale factor $\rho$ is given by  
\begin{align*}
\mathcal{H}_E[\mathbf{E}^\infty_\rho(\cdot;\theta,\theta^\perp)](\bz)
=
-\frac{i\rho^3\K^3a_1^\epsilon}{4\pi}\int_{\mathbb{S}^2} \left(\mathbf{I}_3-\hbx\hbx^\top\right)
\mathbf{M}_D^\epsilon(\theta\times\theta^\perp) e^{i\K(\theta-\hbx)^\top\bz_D}e^{i\K\hbx^\top\bz}ds(\hbx) +O(\rho^4)
\end{align*}
for all $\bz\in\RR^3$ by virtue of expansion \eqref{AE1} together with the assumption $\mu_1=\mu_0$. On  substituting the expression \eqref{E0def}   the above expansion can be written as
\begin{align}
\nonumber
\mathcal{H}_E
[\mathbf{E}^\infty_\rho(\cdot;\theta,\theta^\perp)](\bz)
=&
-\frac{\rho^3\K^2a_1^\epsilon}{4\pi}\int_{\mathbb{S}^2} \left(\mathbf{I}_3-\hbx\hbx^\top\right)e^{i\K\hbx^\top(\bz-\bz_D)}
\mathbf{M}_D^\epsilon\mathbf{E}_0(\bz_D)ds(\hbx) +O(\rho^4),
\\
\nonumber
=&
-\frac{\rho^3\K^2a_1^\epsilon}{4\pi}\left[\left(\mathbf{I}_3+\frac{1}{\K^2}\nabla_\bz\nabla_\bz^\top\right)\int_{\mathbb{S}^2} e^{i\K\hbx^\top(\bz-\bz_D)}ds(\hbx)\right]
\mathbf{M}_D^\epsilon\mathbf{E}_0(\bz_D)
\\
\nonumber
&+O(\rho^4).
\end{align}
Recall that
$$
\int_{\mathbb{S}^2}e^{i\K\hbx^\top(\bz-\bz_D)}ds(\hbx)= \frac{4\pi}{\K}{\Im m\{g(\bz,\bz_D)\}},
$$
and consequently
\begin{eqnarray}
\left(\mathbf{I}_3+\frac{1}{\K^2}\nabla_\bz\nabla_\bz^\top\right)\int_{\mathbb{S}^2} e^{i\K\hbx^\top(\bz-\bz_D)}ds(\hbx)= -\frac{4\pi}{\K\epsilon_0}\Im m\Big\{\mathbf{\Gamma}(\bz,\bz_D)\Big\}.\label{exp1} 
\end{eqnarray}
Therefore  the Herglotz field in the case of a dielectric inclusion can be expressed as 
\begin{align}
\mathcal{H}_E[\mathbf{E}^\infty_\rho(\cdot; \theta,\theta^\perp)](\bz)=\frac{\rho^3\K a_1^\epsilon}{\epsilon_0}\Im m\Big\{\mathbf{\Gamma}(\bz,\bz_D)\Big\}\mathbf{M}^\epsilon_D\mathbf{E}_0(\bz_D)+ O(\rho^4).\label{HEeps}
\end{align}
Finally the asymptotic expansion  
\begin{align}
\Itd[\mathbf{E}_0](\bz_S)
=
\ds - \frac{\rho^3\K^3a_1^\epsilon a_2^\epsilon}{4\pi\epsilon_0}\Re e\left\{
\Im m\Big\{\mathbf{\Gamma}(\bz_S,\bz_D)\Big\}\mathbf{M}_D^\epsilon\overline{\mathbf{E}_0(\bz_D)}
\cdot \mathbf{M}_S^\epsilon
\mathbf{E}_0(\bz_S)\right\} +O(\rho^4),
\label{TDeps2}
\end{align}  
of the indicator function for a dielectric inclusion is obtained by injecting back the expansion \eqref{HEeps} of the Herglotz field  in \eqref{TDeps}.

\subsubsection{Permeability contrast}

Suppose that the contrast of $D$ is defined by the permeability $\mu_1$  which is different from background permeability $\mu_0$ whereas $\epsilon_1=\epsilon_0$. As in the previous case, the trial inclusion is created with the background permittivity $\epsilon_0$, that is, $\epsilon_2=\epsilon_0$. Therefore  \eqref{TD} reduces to
\begin{align}
\Itd[\mathbf{E}_0](\bz_S)
=
-\ds\frac{a_2^\mu}{4\pi}
\Re e\left\{
\mathbf{M}_S^\mu\nabla \times \mathbf{E}_0(\bz_S)\cdot\nabla \times
\overline{\mathcal{H}_E[\mathbf{E}^\infty_\rho(\cdot,\theta,\theta^\perp)](\bz_S)}\right\}.
\label{TDmu}
\end{align}  
By virtue of Theorem \ref{Asymp} and expression \eqref{CurlE} the asymptotic expansion of $\mathcal{H}_E[\mathbf{E}^\infty_\rho]$ versus $\rho$ for a permeable inclusion turns out to be 
\begin{align}
\nonumber
\mathcal{H}_E
[\mathbf{E}^\infty_\rho&(\cdot;\theta,\theta^\perp)](\bz)
\\
\nonumber
=&
-\frac{i\rho^3\K^3 a_1^\mu}{4\pi}\int_{\mathbb{S}^2} 
\left(\mathbf{M}_D^\mu\left(\theta\times (\theta\times\theta^\perp)\right)e^{i\K\theta^\top\bz_D}\right) \times\hbx e^{i\K\hbx^\top(\bz-\bz_D)}ds(\hbx) +O(\rho^4)
\\
\nonumber
=&\phantom{- }
\frac{i\rho^3\K a_1^\mu}{4\pi}\int_{\mathbb{S}^2} 
\left(\mathbf{M}_D^\mu\nabla\times
\mathbf{E}_0(\bz_D)\right)\times\hbx  e^{i\K\hbx^\top(\bz-\bz_D)}ds(\hbx) +O(\rho^4), \qquad\bz\in\RR^3.
\end{align}
Consequently, by invoking identities \eqref{relation2} and \eqref{exp1}, $\nabla\times\mathcal{H}_E[\mathbf{E}^\infty_\rho]$ can be calculated as
\begin{align}
\nonumber
\nabla\times\mathcal{H}_E&
[\mathbf{E}^\infty_\rho(\cdot;\theta,\theta^\perp)](\bz)
\\
\nonumber
=&
-\frac{\rho^3\K^2a_1^\mu}{4\pi}\int_{\mathbb{S}^2} 
\hbx\times \left\{\left(\mathbf{M}_D^\mu\nabla\times
\mathbf{E}_0(\bz_D)\right)\times\hbx\right\}  e^{i\K\hbx^\top(\bz-\bz_D)}ds(\hbx) +O(\rho^4) 
\\
=&
-\frac{\rho^3\K^2 a_1^\mu}{4\pi}\int_{\mathbb{S}^2} 
\left(\mathbf{I}_3-\hbx\hbx^\top\right)e^{i\K\hbx^\top(\bz-\bz_D)}\mathbf{M}_D^\mu\nabla\times
\mathbf{E}_0(\bz_D)ds(\hbx) +O(\rho^4) 
\nonumber
\\
=&
-\frac{\rho^3\K^2 a_1^\mu}{4\pi}\int_{\mathbb{S}^2} 
\left(\mathbf{I}_3+\frac{1}{\K^2}\nabla_\bz\nabla_\bz^\top\right)e^{i\K\hbx^\top(\bz-\bz_D)}\mathbf{M}_D^\mu\nabla
\times
\mathbf{E}_0(\bz_D)ds(\hbx) +O(\rho^4) 
\nonumber
\\
=&
\phantom{- }\frac{\rho^3\K a_1^\mu}{\epsilon_0}\Im m\{\mathbf{\Gamma}(\bz,\bz_D)\}\mathbf{M}_D^\mu\nabla\times
\mathbf{E}_0(\bz_D) +O(\rho^4).\label{HEmu}
\end{align}

By substituting the expansion \eqref{HEmu} into \eqref{TDmu}  the asymptotic expansion of the indicator function $\Itd$ for a permeable inclusion is obtained by 
\begin{align}
\Itd[\mathbf{E}_0](\bz_S)
=
-\ds\frac{\rho^3\K a_1^\mu a_2^\mu}{4\pi\epsilon_0}  
\Re e\left\{
\Im m\big\{\mathbf{\Gamma}(\bz_S,\bz_D)\big\} 
\mathbf{M}_D^\mu\overline{\nabla\times
\mathbf{E}_0(\bz_D)}\cdot \mathbf{M}_S^\mu\nabla\times \mathbf{E}_0(\bz_S)\right\}+O(\rho^4).
\label{TDmu2}
\end{align}

\subsubsection{Decay and sign heuristics for location indicator function} \label{ss:sign}

In order to determine the ability of $\Itd$ to locate an inclusion, it is important to debate its decay and sign characteristics. It is evident  from expansions \eqref{TDeps2} and \eqref{TDmu2} and the expression for the imaginary part of the fundamental solution that 
\begin{align*}
\Itd[\mathbf{E}_0](\bz_S)
\propto &
\phantom{-.}\Im m\big\{\mathbf{\Gamma}(\bz_S,\bz_D)\big\}\overline{\mathbf{p}^\gamma_D}\cdot\mathbf{p}^\gamma_S 
\\
=&  -\frac{\epsilon_0\K}{4\pi}\left[\frac{2}{3}j_0(\K r)\mathbf{I}_3+j_2(\K r)\left(\widehat{\mathbf{r}}\,\widehat{\mathbf{r}}^\top-\frac{1}{3}\mathbf{I}_3\right)\right]\overline{\mathbf{p}^\gamma_D}\cdot\mathbf{p}^\gamma_S,
\quad \gamma\in\{\mu,\epsilon\},
\end{align*}
for a search point $\bz_S\in\RR^3$ where 
$$
\mathbf{r}:=\bz_S-\bz_D, \quad 
r:=|\mathbf{r}|, 
\quad 
\widehat{\mathbf{r}}=\frac{\mathbf{r}}{r},
\quad
\mathbf{p}_{i}^\epsilon=\mathbf{M}^{\epsilon}_{i}\mathbf{E}_0(\bz_{i}),
\quad 
\mathbf{p}_{i}^\mu=\mathbf{M}^\mu_{i}\nabla\times\mathbf{E}_0(\bz_i),
\quad
i\in\{D,S\}.
$$ 
The function $j_n$ represents an order $n$ spherical Bessel function of first kind. Recall that $j_n(kr)$ behaves like $O(1/kr)$ when $kr\to\infty$, and like $O((kr)^n)$ when $kr\to 0$ (see, for instance, \cite[10.52.3]{NIST}). Therefore, the indicator function $\bz_S\mapsto \Itd[\mathbf{E}_0](\bz_S)$ rapidly decays for $\bz_S$ away from $\bz_D$ and has a sharp peak with a focal spot size of the order of half the operating wavelength when $\bz_S\to\bz_D$.  

Heuristically, if the contrasts $a_2^\mu$ and $a_2^\epsilon$ have the same signs as $a_1^\mu$ and $a_2^\epsilon$ respectively, then $\mathcal{J}[\mathbf{E}_0](\bz_S)$ will observe the most pronounced decrease at the potential candidate $\bz_S$ of the true location $\bz_D$. In other words, the indicator function $\bz_S\mapsto\Itd[\mathbf{E}_0](\bz_S)$  should attain its most pronounced positive value (being defined as the negative of the topological derivative of $\mathcal{J}[\mathbf{E}_0](\bz_S)$) when $\bz_S\to\bz_D$. Refer, for instance, to \cite{Bellis, BG} for detailed discussions on sign heuristic. 

Remark that if the contrasts of true and trial inclusions have same signs then $\rho^3 a_1^\mu a_2^\mu\geq 0$ and $\rho^3 a_1^\epsilon a_2^\epsilon\geq 0$ regardless of whether $D$ has only permeability contrast or only permittivity contrast.  Thus  $\bz_S\mapsto\Itd[\mathbf{E}_0](\bz_S)$ indeed achieves its maximum positive value when $\bz_S\to\bz_D$ due to its decay properties. Therefore, it synthesizes the sensitivity of cost functional $\mathcal{J}[\mathbf{E}_0](\bz_S)$ relative to the insertion of an inclusion at search point $\bz_S\in\RR^3$.
In a nutshell, $\Itd[\mathbf{E}_0]$ is felicitous to detect $\bz_D$ with Rayleigh's resolution.

In order to back our theoretical findings, we perform a few numerical experiments for the detection of spherical inclusions with only permittivity contrast and permeability contrast. The inclusions are assumed to be centered at the origin with $\rho=0.01$. The incident fields are excited at wave-numbers $\K=4\pi$ and $\K=8\pi$ with incident direction $\theta=(1,0,0)^\top$ and corresponding perpendicular $\theta^\perp=(0,1,0)^\top$.   Figures \ref{fig:1} and \ref{fig:2} illustrate the maps of  $\bz_S\mapsto\mathcal{I}_{TD}[\mathbf{E}_0](\bz_S)$ over $[-1,1]\times[-1,1]\times\{0\}$ for an inclusion having permittivity contrast with and without deterministic measurement noise respectively.  In Figures \ref{fig:3} and \ref{fig:4} we present similar results for permeability contrast case. In both cases we use $201\times 201$ sampling points. A sharp peak can be observed in all cases indicating the locations of the inclusions. Moreover, the focal spot size of the location indicator function  clearly decreases with an increase in the excitation frequency in both permittivity and permeability contrast cases as per Rayleigh resolution criterion. In order to see the effects of noise on the detection capabilities of the location indicator function we choose $10\%$ and $20\%$ deterministic noise corrupting the measurements when the incident waves are excited at wave-number  $\K=4\pi$. The effect of a more general Gaussian random noise will be quantitatively discussed in detail in Section \ref{measNoise}.
\begin{figure}[!ht]
\includegraphics[width=0.43\textwidth]{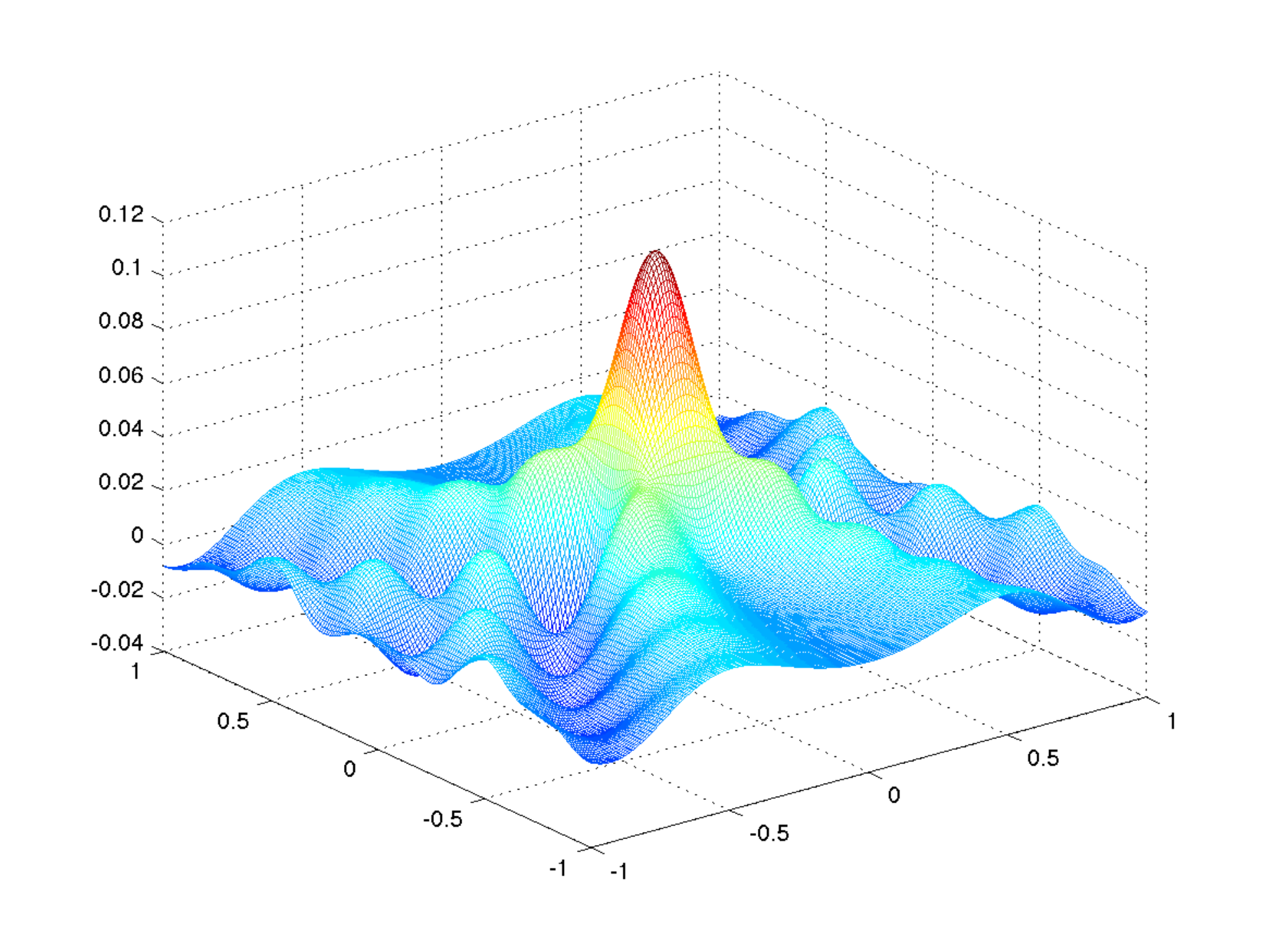}
\includegraphics[width=0.43\textwidth]{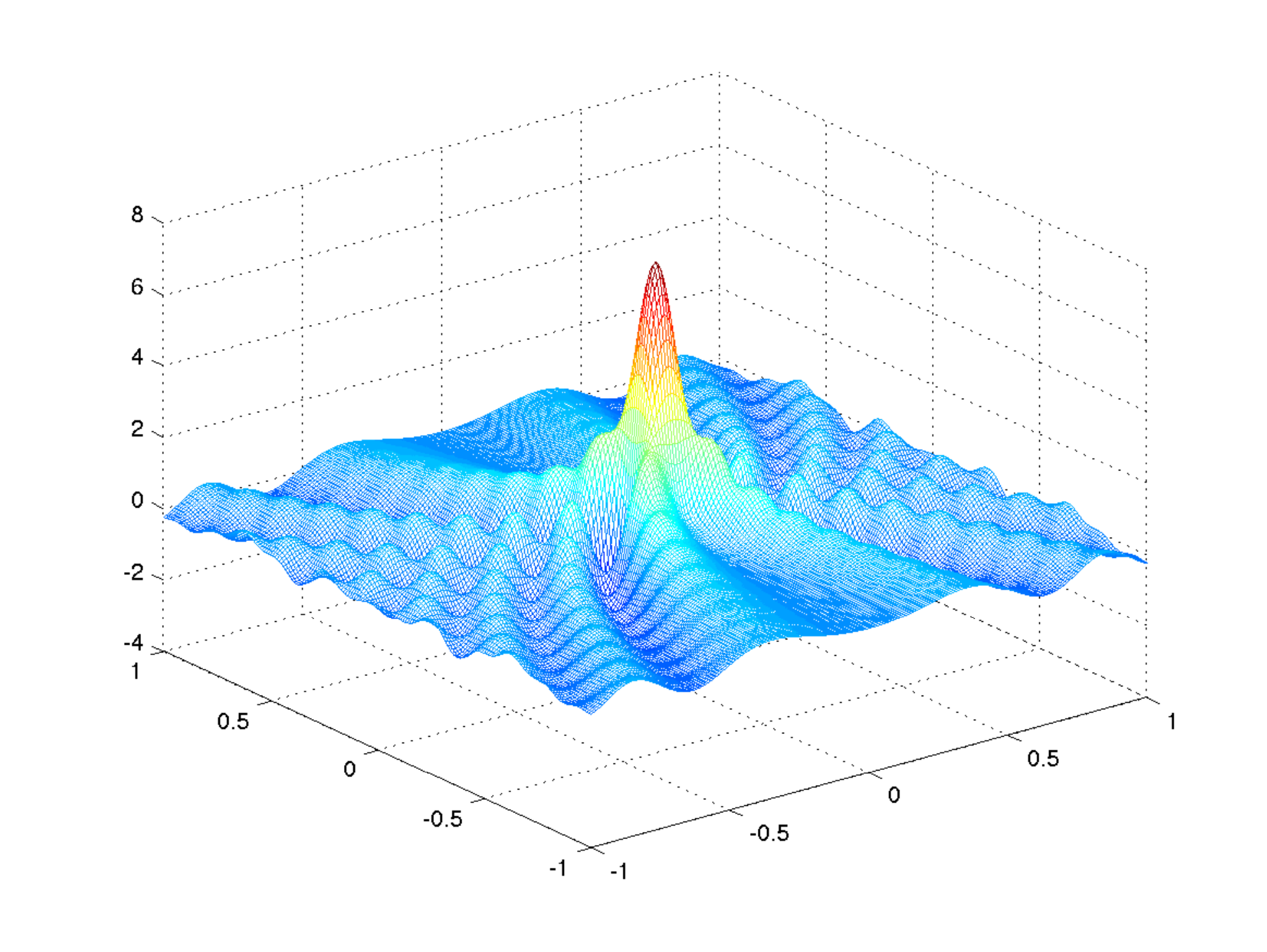}
\includegraphics[width=0.48\textwidth]{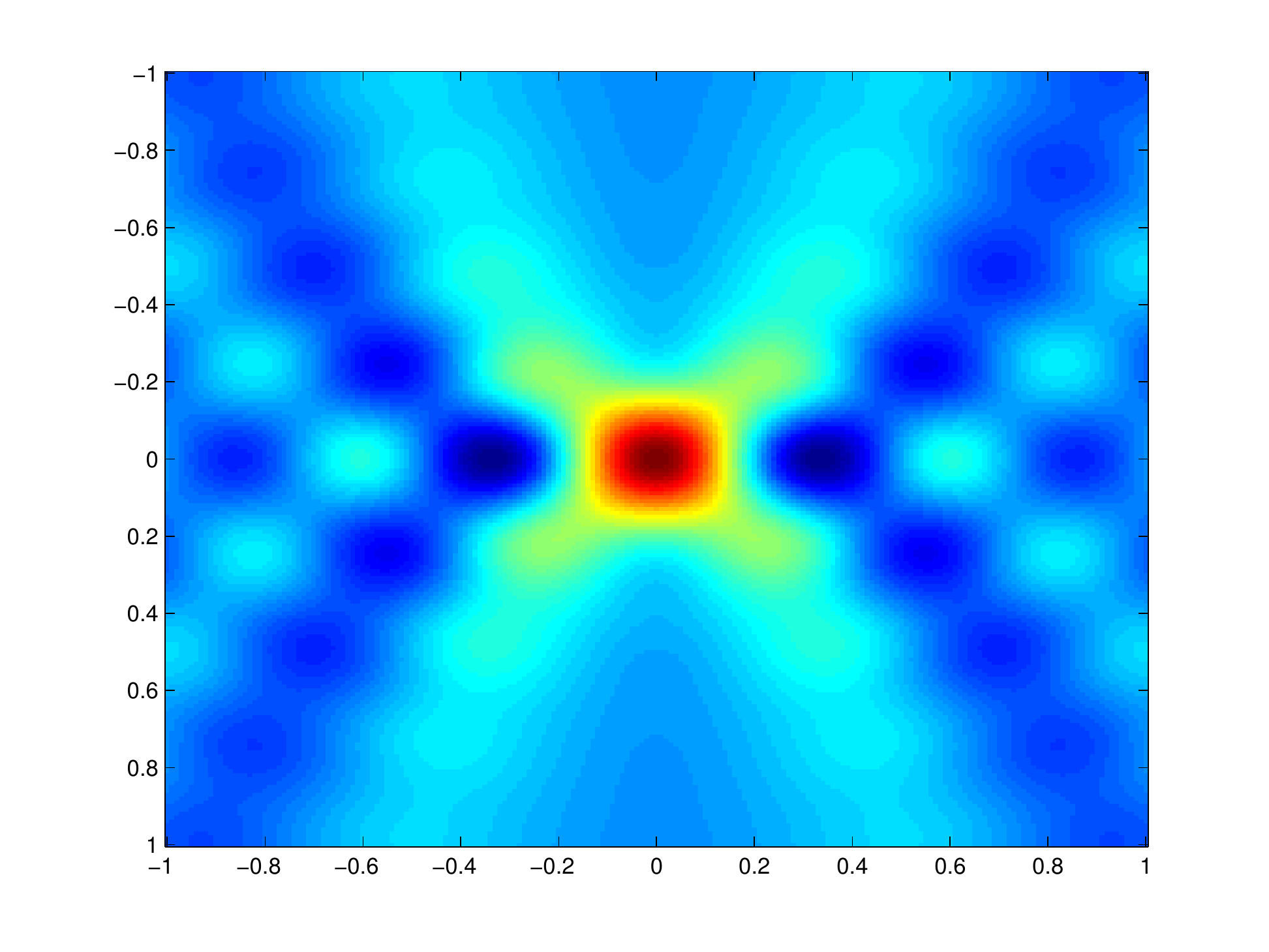}
\includegraphics[width=0.48\textwidth]{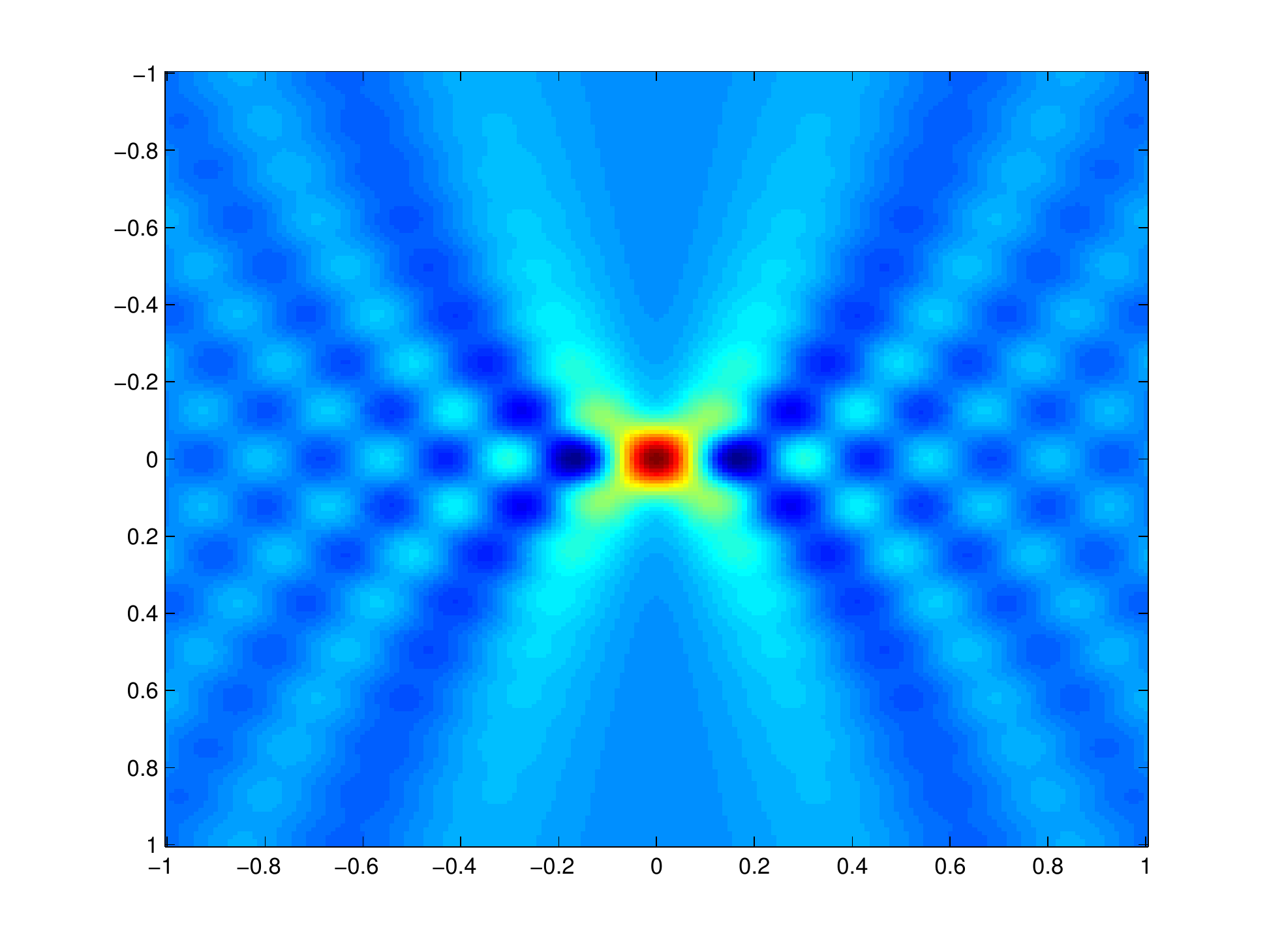}
\caption{Detection of a spherical inclusion with permittivity contrast ($\epsilon_0=1$, $\epsilon_1=2$, $\mu_0=1$ and $\mu_1=1$). Left: $\K=4\pi$. Right: $\K=8\pi$. Top: $\bz_S\mapsto\Itd[\mathbf{E}_0](\bz_S)$. Bottom: Focal spot.}\label{fig:1}
\end{figure}
\begin{figure}[!ht]
\includegraphics[width=0.43\textwidth]{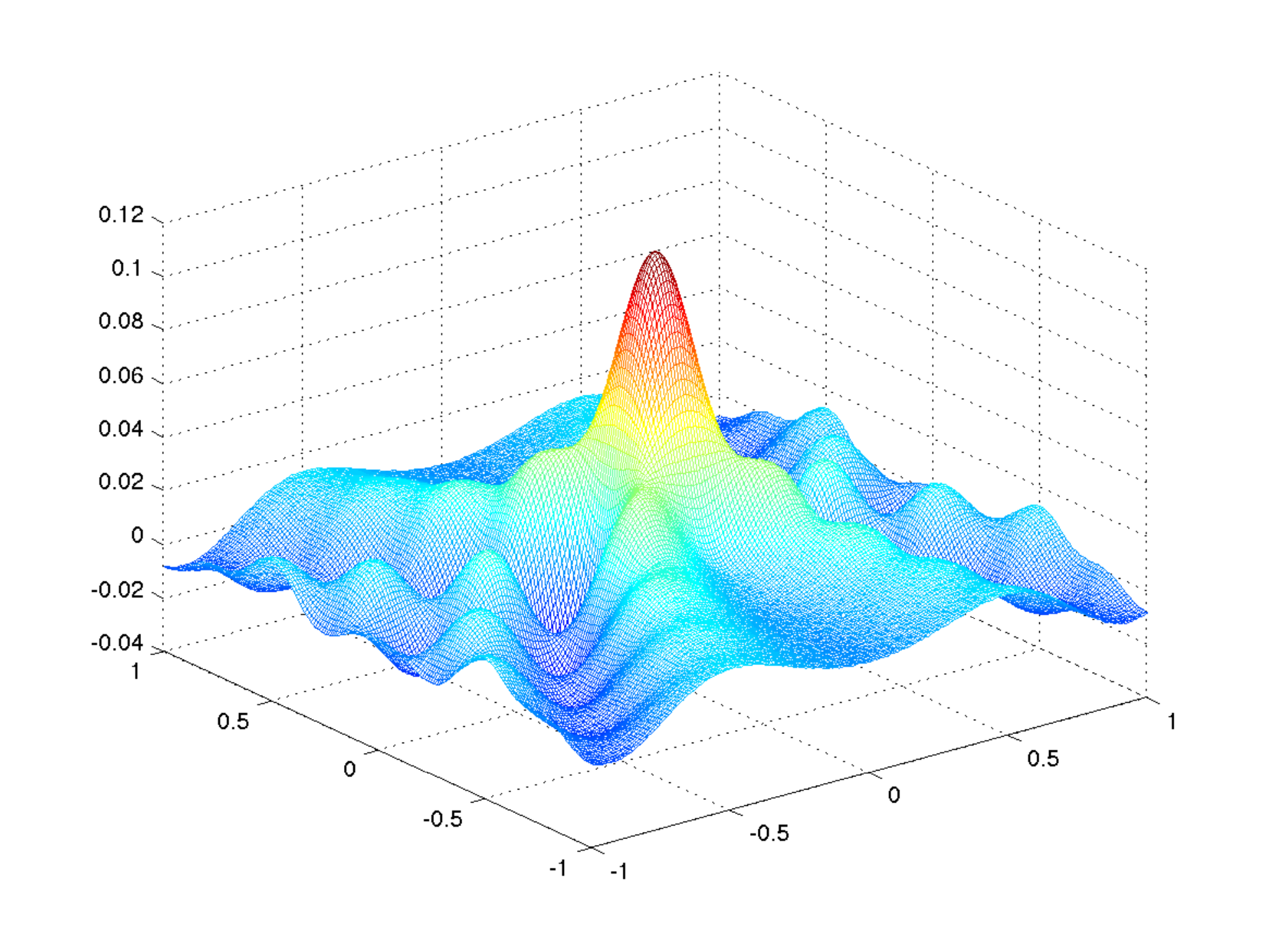}
\includegraphics[width=0.43\textwidth]{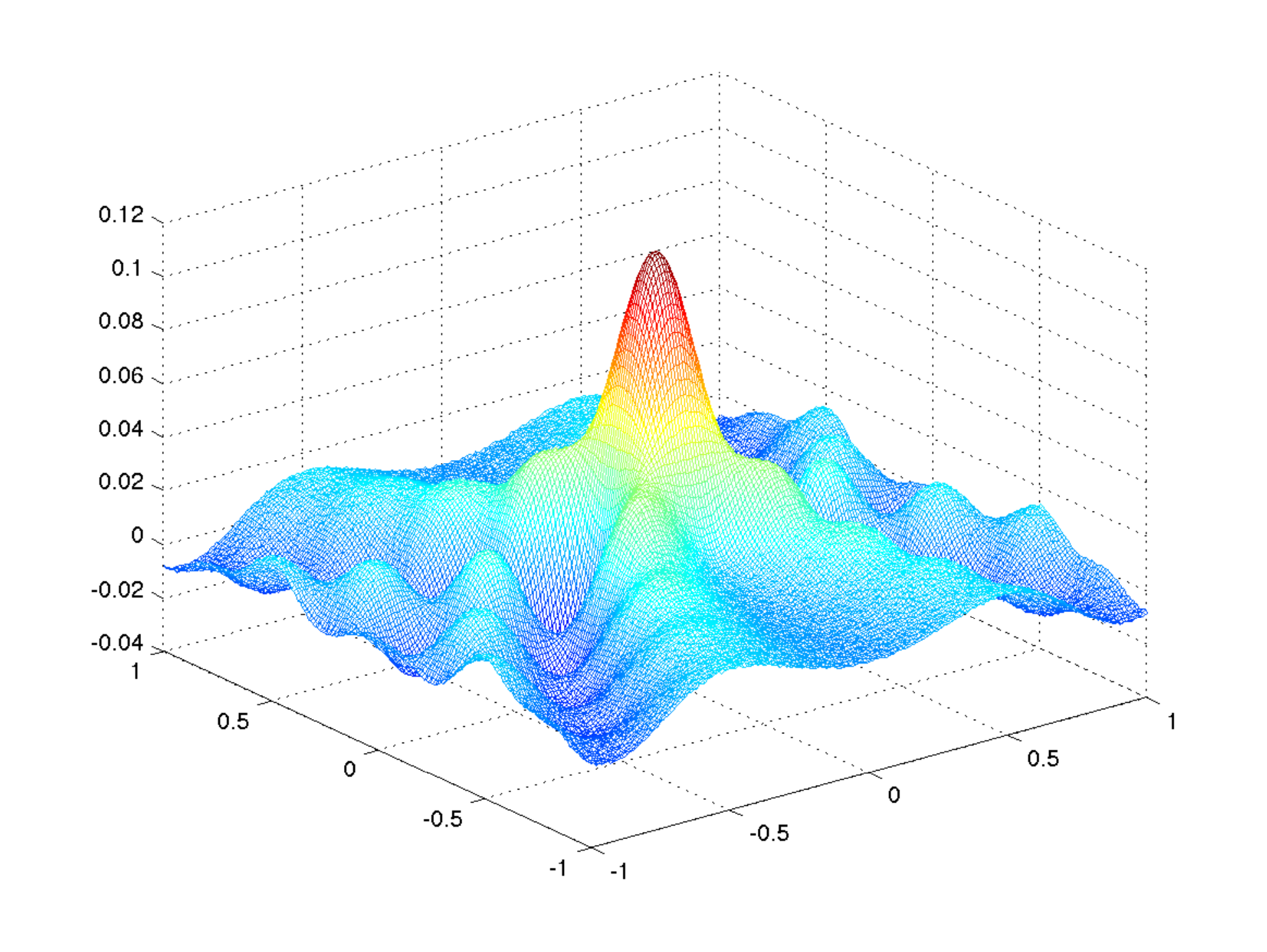}
\includegraphics[width=0.48\textwidth]{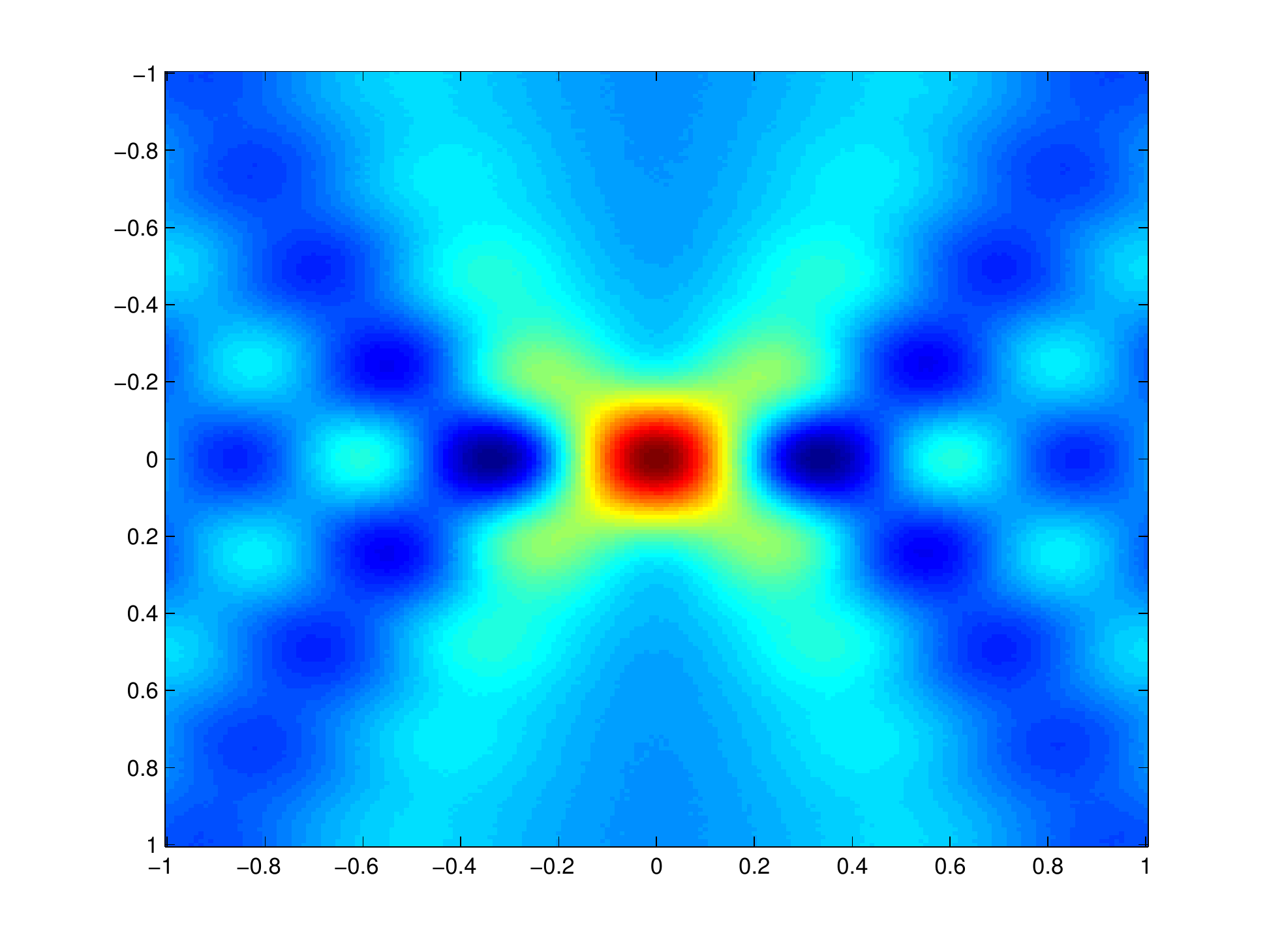}
\includegraphics[width=0.48\textwidth]{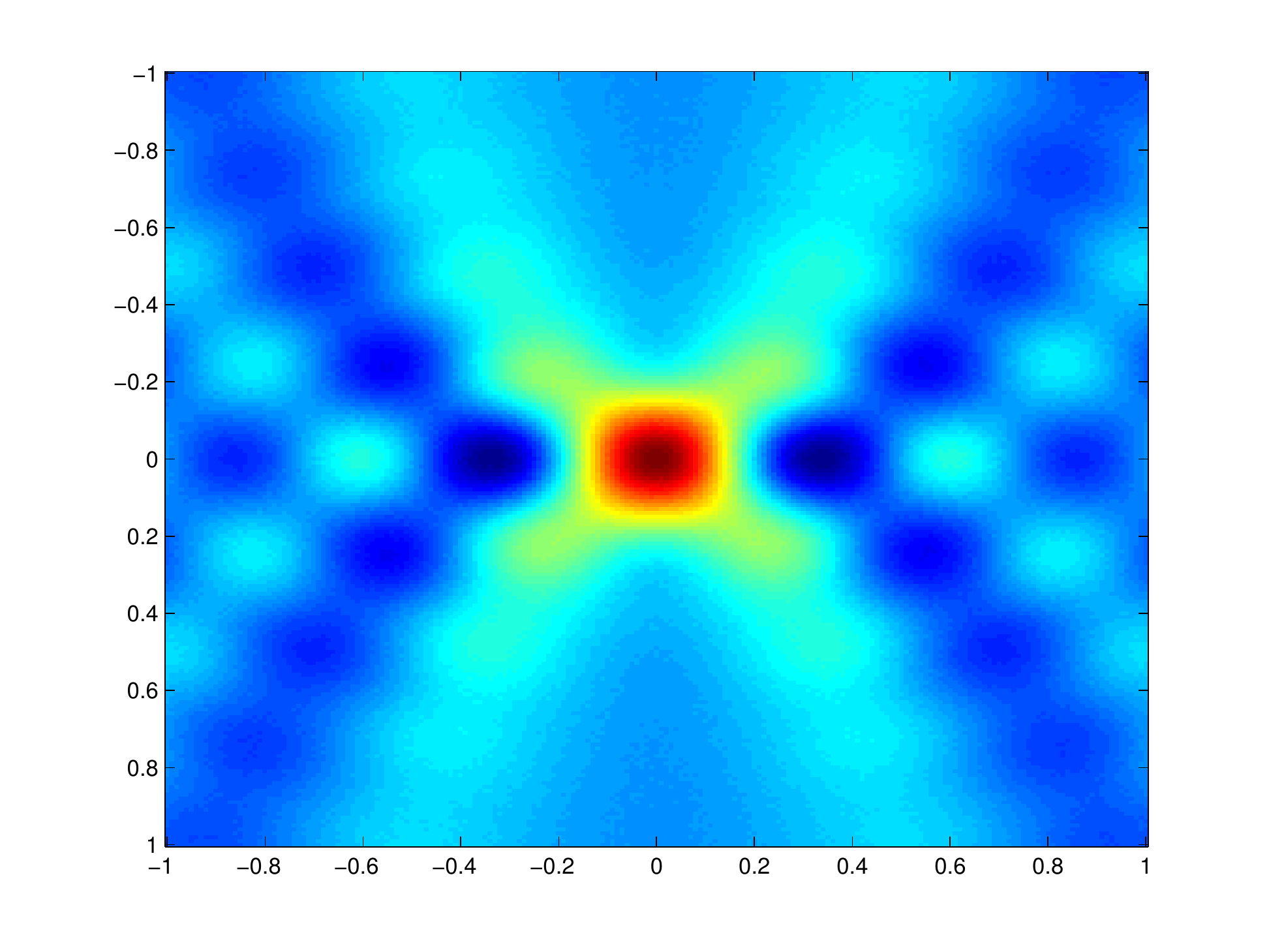}
\caption{Detection of a spherical inclusion with permittivity contrast ($\epsilon_0=1$, $\epsilon_1=2$, $\mu_0=1$ and $\mu_1=1$) at $\K=4\pi$. Left: $10\%$ noise. Right: $20\%$ noise. Top: $\bz_S\mapsto\Itd[\mathbf{E}_0](\bz_S)$. Bottom: Focal spot.}\label{fig:2}
\end{figure}
\begin{figure}[!ht]
\includegraphics[width=0.43\textwidth]{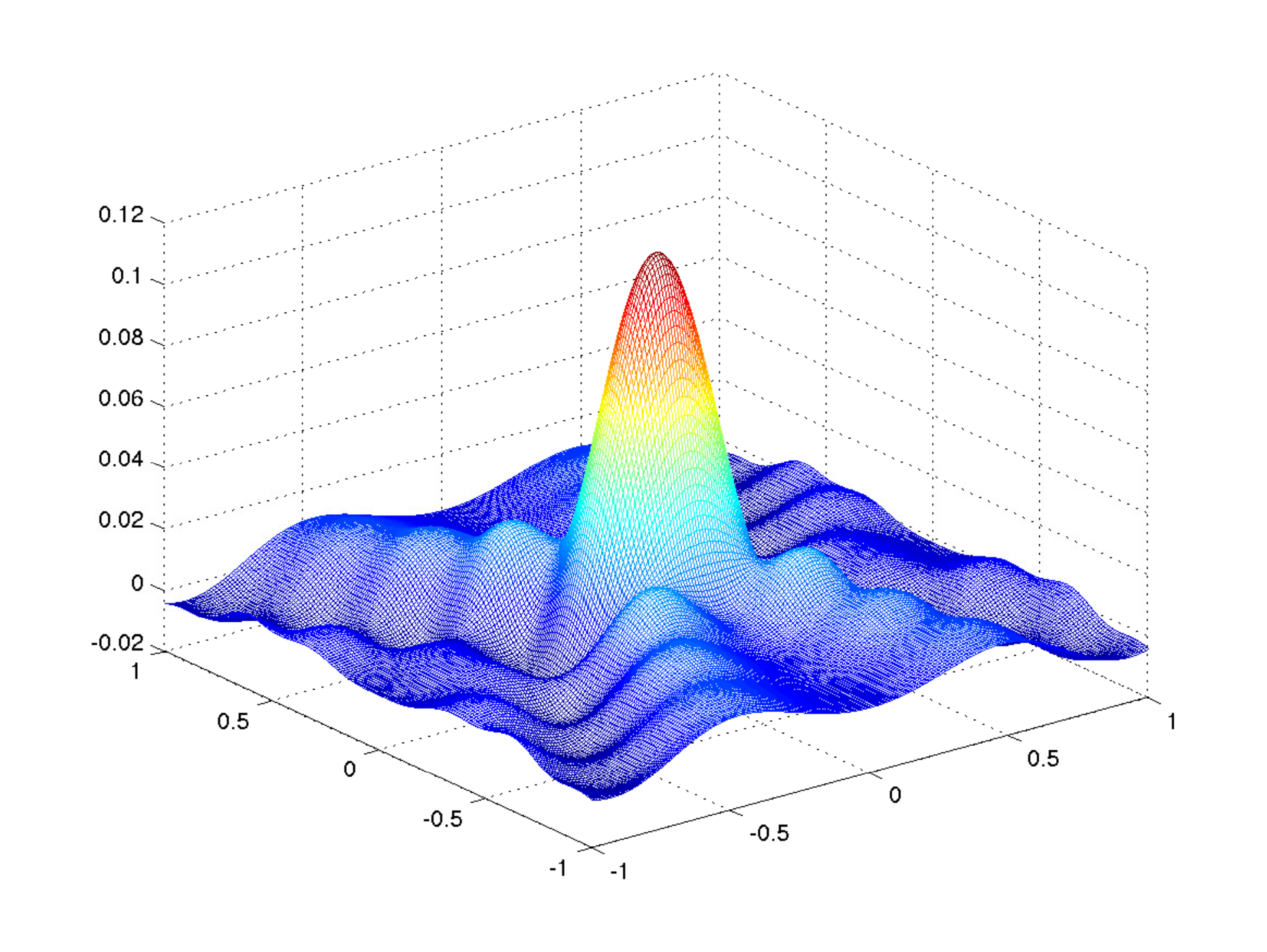}
\includegraphics[width=0.43\textwidth]{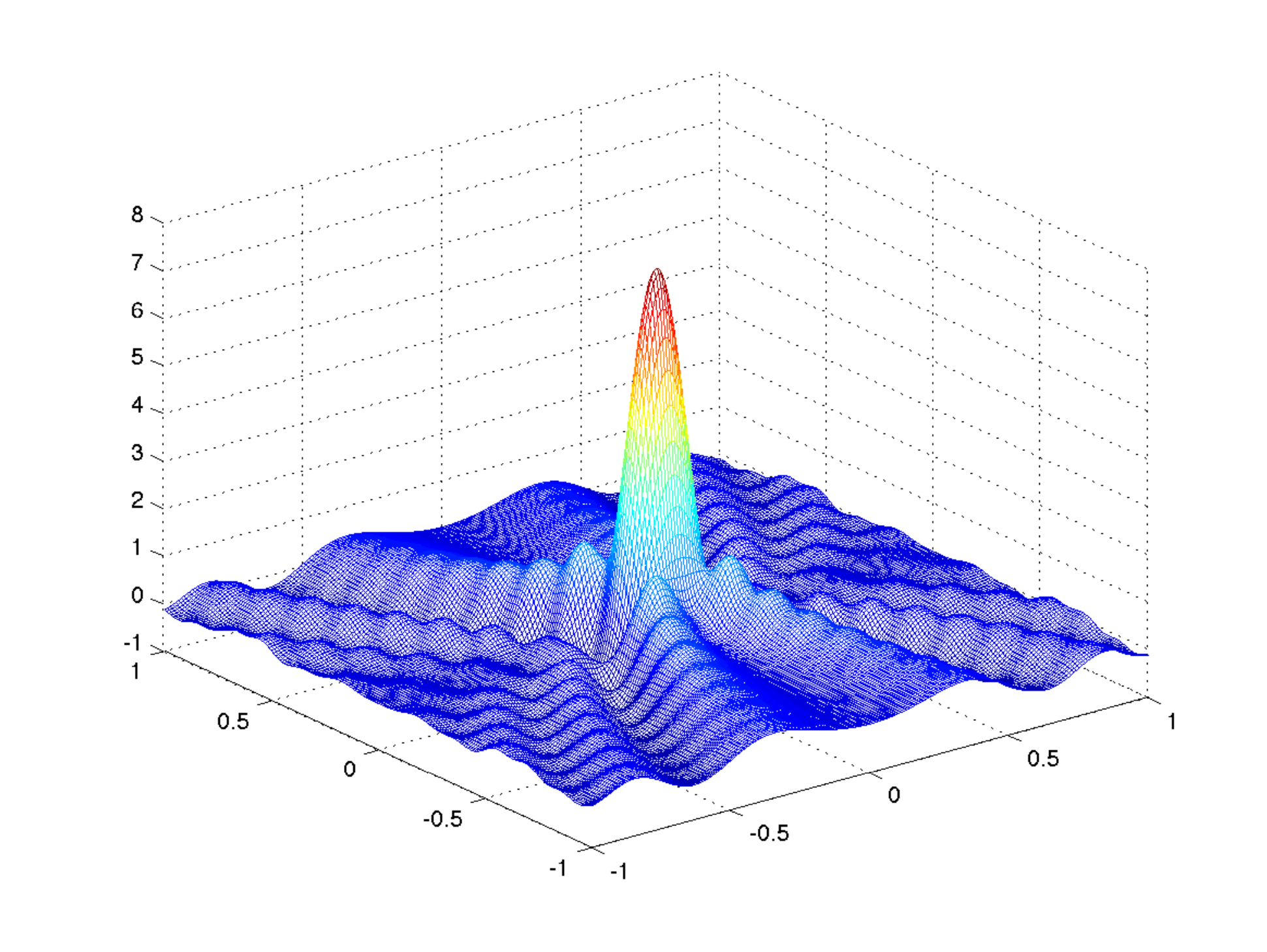}
\includegraphics[width=0.48\textwidth]{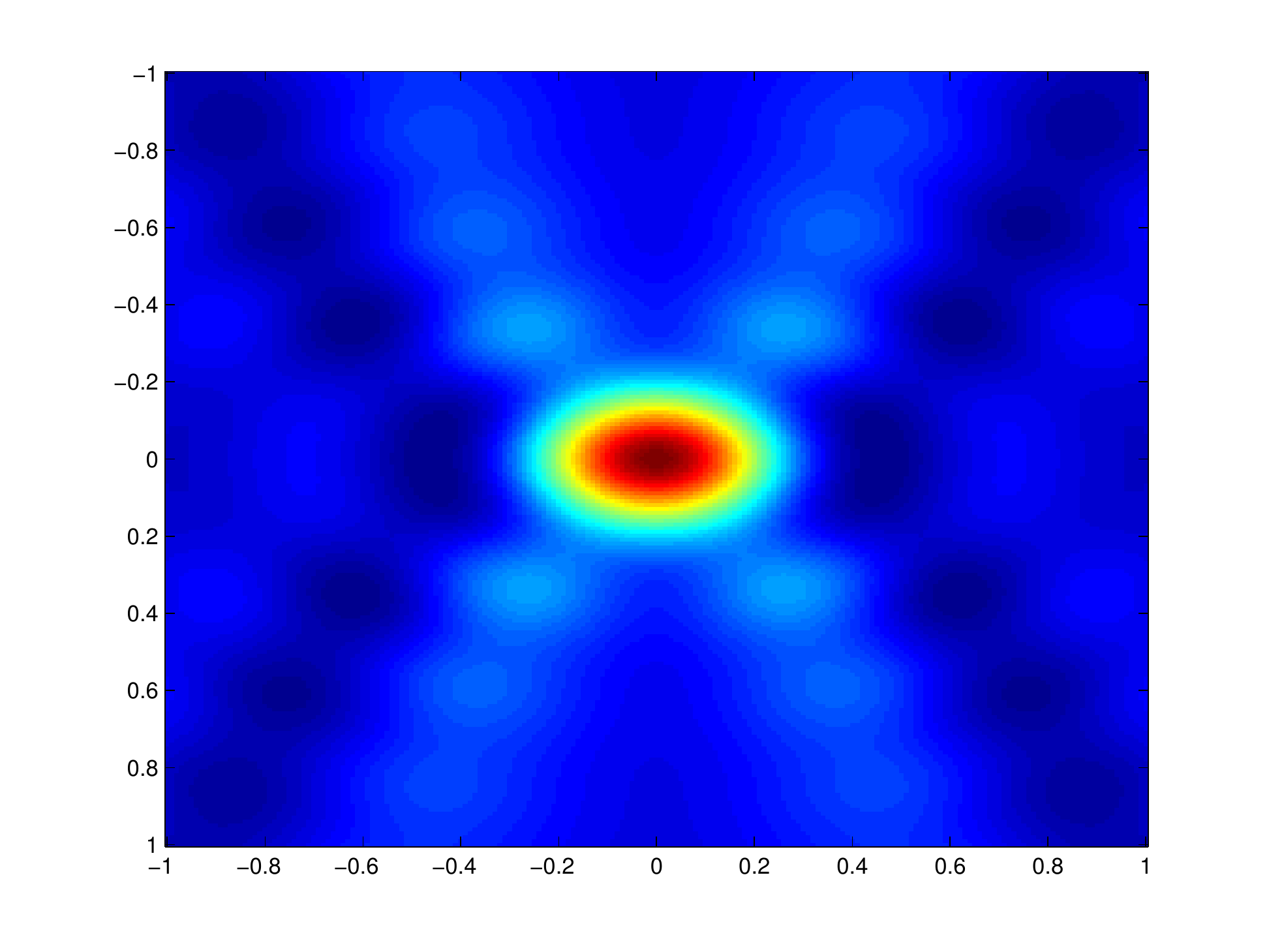}
\includegraphics[width=0.48\textwidth]{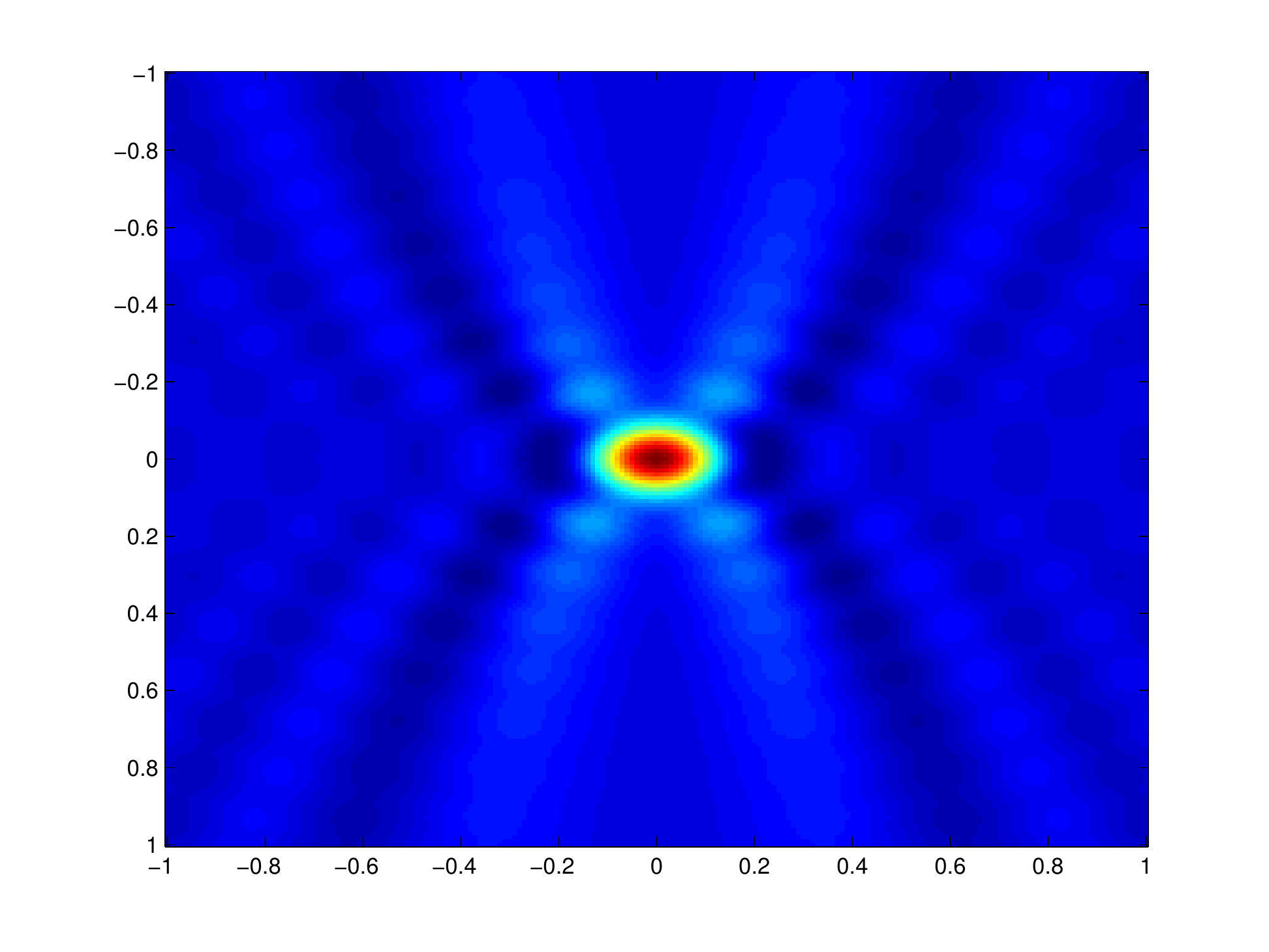}
\caption{Detection of spherical inclusion with permeability contrast ($\epsilon_0=1=\epsilon_1$, $\mu_0=1$ and $\mu_1=2$). Left: $\K=4\pi$. Right: $\K=8\pi$. Top: $\bz_S\mapsto\Itd[\mathbf{E}]_0(\bz_S)$. Bottom: Focal spot.}\label{fig:3}
\end{figure} 
\begin{figure}[!ht]
\includegraphics[width=0.43\textwidth]{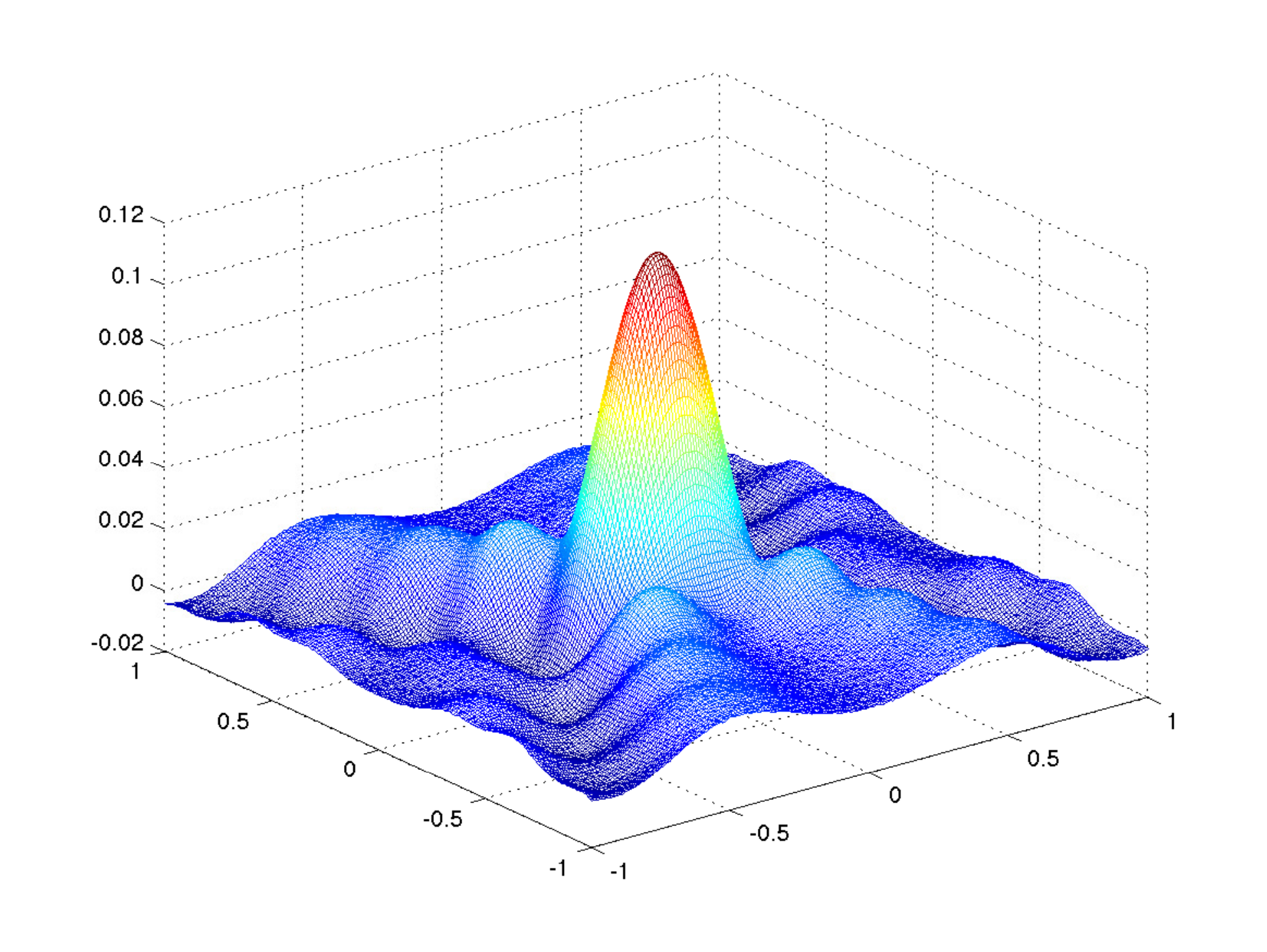}
\includegraphics[width=0.43\textwidth]{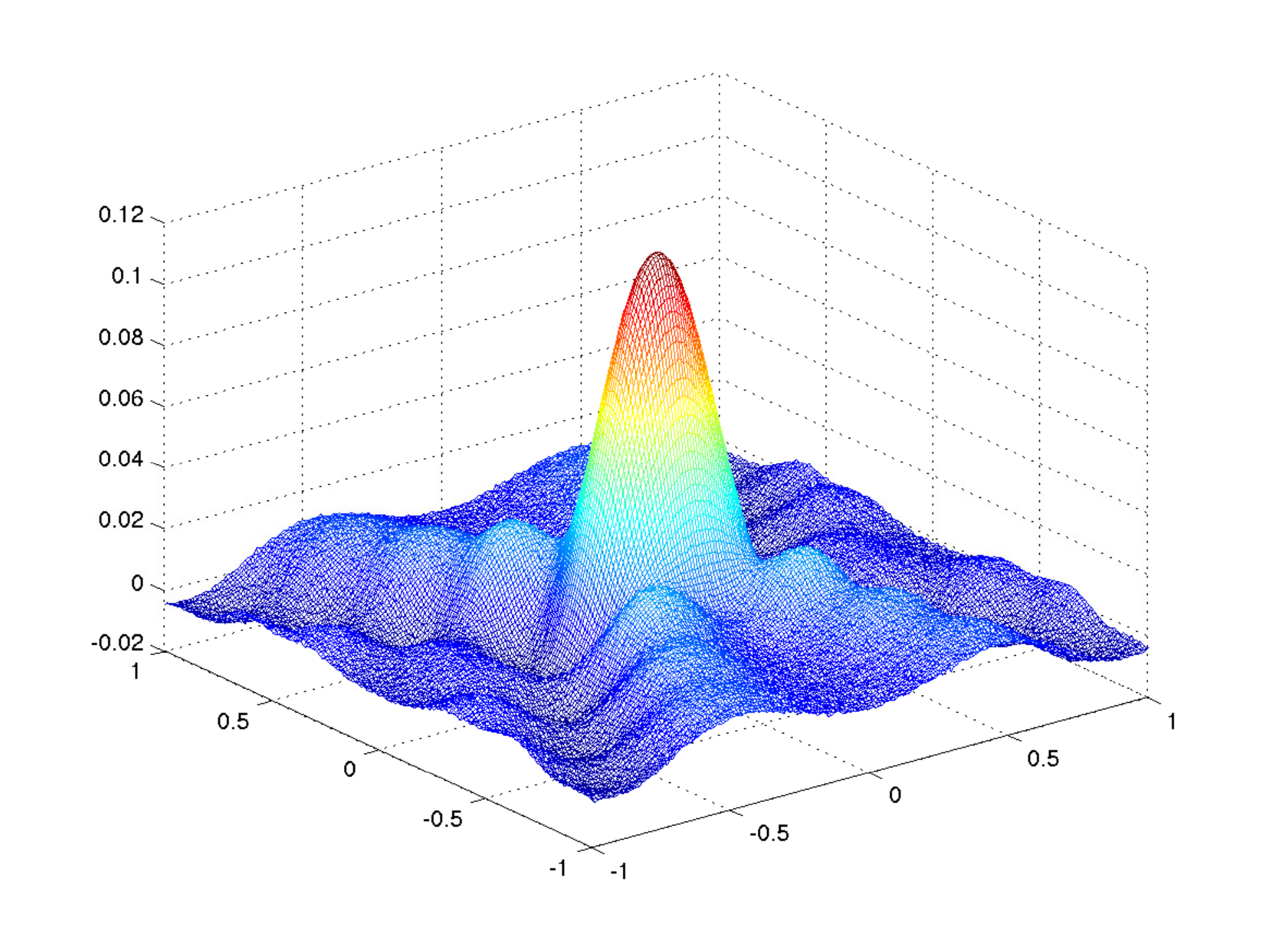}
\includegraphics[width=0.48\textwidth]{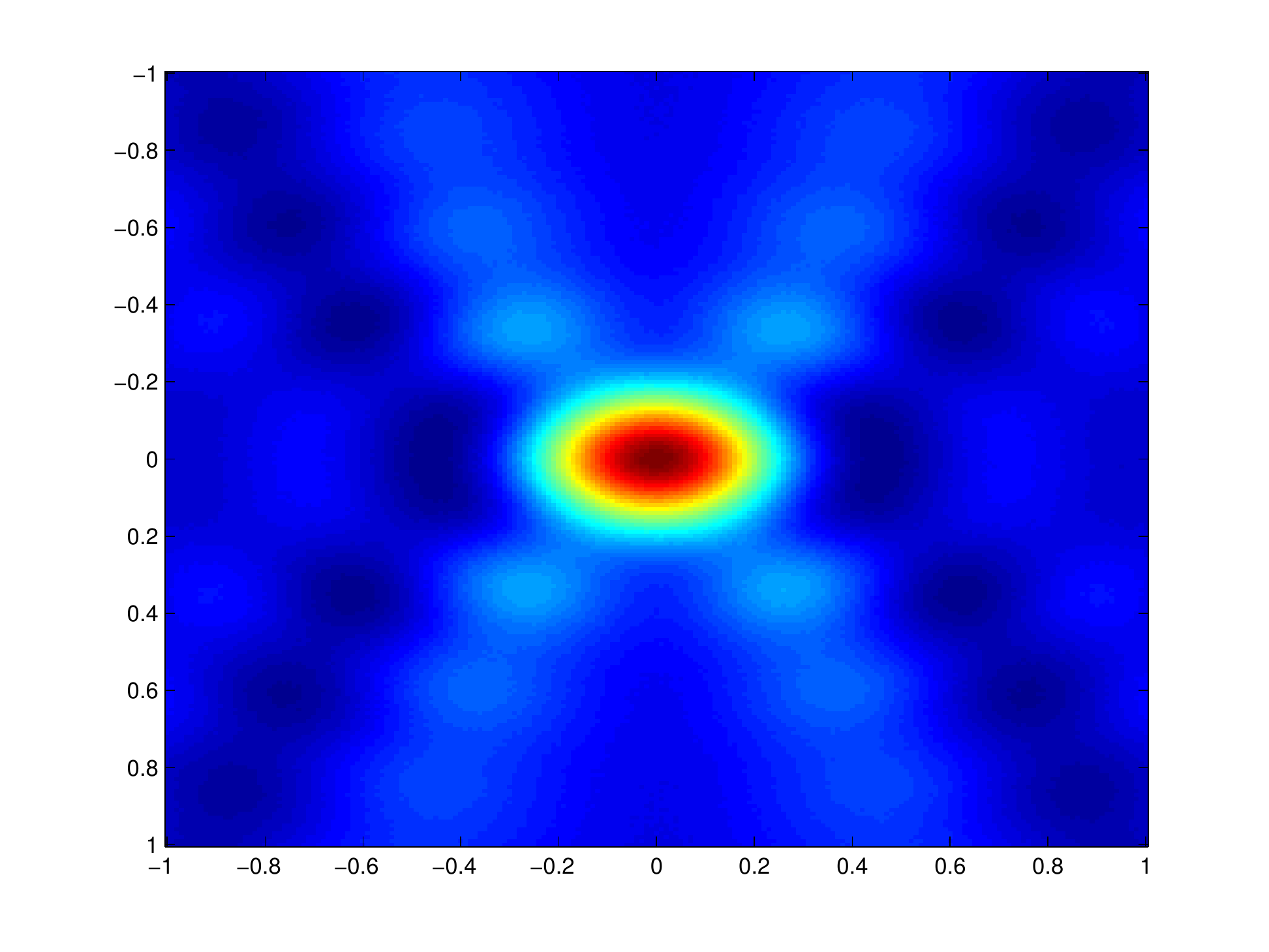}
\includegraphics[width=0.48\textwidth]{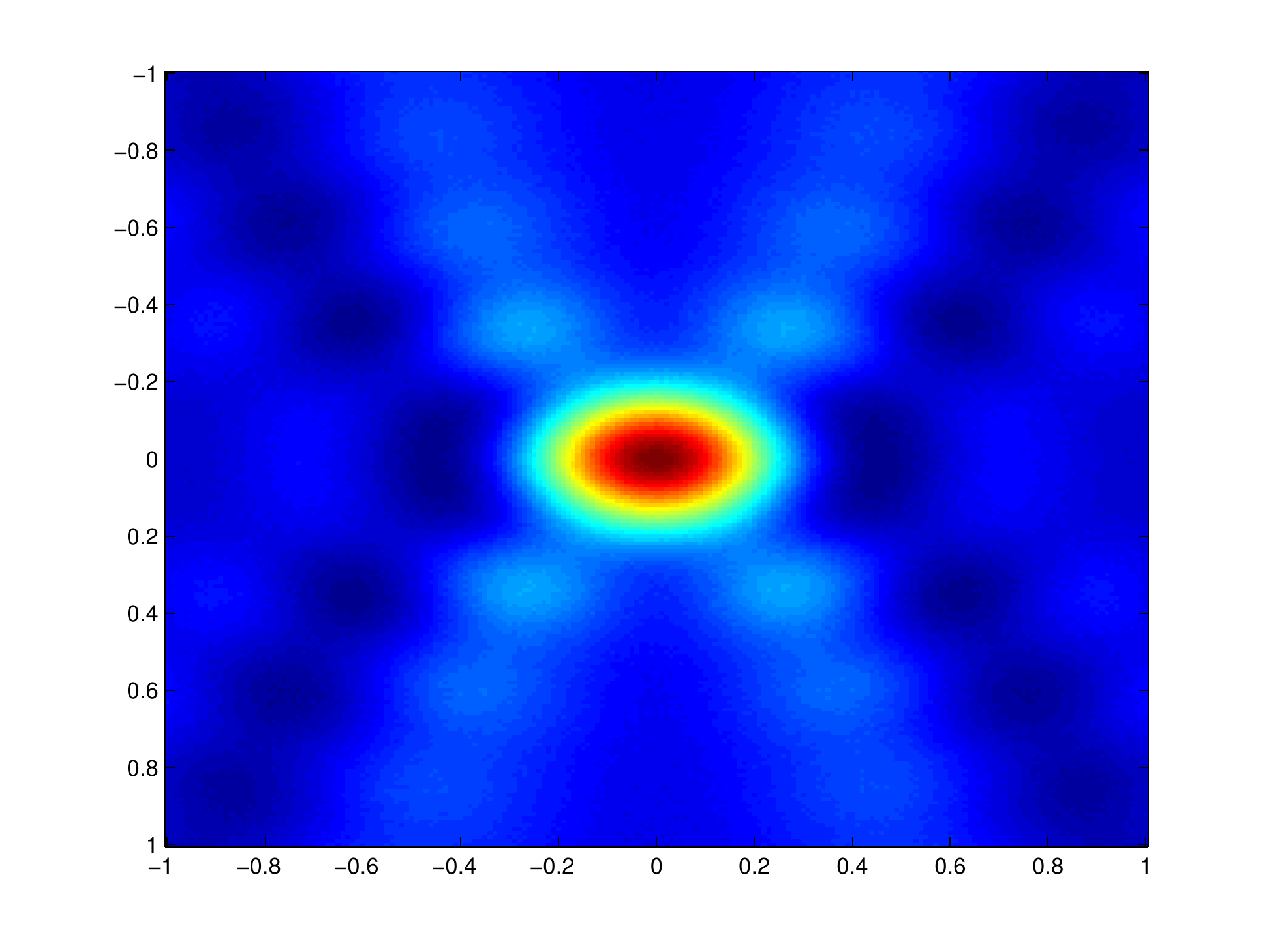}
\caption{Detection of spherical inclusion with permeability contrast ($\epsilon_0=1=\epsilon_1$, $\mu_0=1$ and $\mu_1=2$) at $\K=4\pi$. Left: $10\%$ noise. Right: $20\%$ noise. Top: $\bz_S\mapsto\Itd[\mathbf{E}]_0(\bz_S)$. Bottom: Focal spot.}\label{fig:4}
\end{figure} 

\subsection{Inclusion detection using multiple measurements}

The analysis in the previous section elucidates the appositeness of the topological derivative based indicator function $\Itd$ with a single measurement. In this section  an indicator function is built from $\Itd$ for the case of multiple measurements. Practically, the enriched data set with multiple measurements is expected to enhance the performance of the detection function in terms of stability and signal-to-noise ratio.   

Let $\theta_1, \theta_2, \cdots ,\theta_n\in\mathbb{S}^2$  be $n $ equi-distributed directions for some $n\in\mathbb{N}$ and $\theta_j^{\perp,\ell}$ ($\ell=1,2$) be such that $\{\theta_j,\theta_j^{\perp,1},\theta_j^{\perp,2}\}$ forms an orthonormal basis of $\RR^3$ for each $j\in\{1,\cdots, n\}$. Consider $n$ incident electric fields satisfying \eqref{E0}  by
\begin{equation}\label{Ej}
\mathbf{E}_0^{j,\ell}(\bx):=i\K\theta_j\times\theta_j^{\perp,\ell}e^{i\K\theta^\top_j\bx},\qquad \forall\,\bx\in\RR^3,
\end{equation}
and construct the topological sensitivity based  multi-measurement location indicator function by
\begin{equation}\label{MultiFunc}
\Itd^n(\bz_S):= \frac{1}{n}\sum_{\ell=1}^2\sum_{j=1}^n\Itd[\mathbf{E}_0^{j,\ell}](\bz_S),\qquad\bz_S\in\RR^3. 
\end{equation}

Before further discussion, note that 
$$
\theta_j^{\perp,1}\left(\theta_j^{\perp,1}\right)^\top+\theta_j^{\perp,2}\left(\theta_j^{\perp,2}\right)^\top=\left(\mathbf{I}_3-\theta_j\theta_j^\top\right),
\qquad j=1,2\cdots,n,
$$
since $\big\{\theta_j,\theta_j^{\perp,1},\theta_j^{\perp,2}\big\}$ is a basis of $\RR^3$. Moreover, for $n$ sufficiently large 
$$
\frac{1}{n}\sum_{j=1}^n e^{i\K\theta_j^\top(\bx-\by)}\approx j_0(\K|\bx-\by|)=\frac{4\pi}{\K}\Im m\{g(\bx,\by)\},
\qquad\forall\, \bx,\by\in\RR^3,\, \bx\neq \by.
$$
Thus
\begin{align}
\nonumber
\frac{1}{n}\sum_{\ell=1}^2\sum_{j=1}^n e^{i\K\theta_j^\top(\bx-\by)} 
\theta_j^{\perp,\ell} \left(\theta_j^{\perp,\ell}\right)^\top
=&
\frac{1}{n}\sum_{j=1}^n \left(\mathbf{I}_3-\theta_j\theta_j^\top\right)
 e^{i\K\theta_j^\top(\bx-\by)}
\\
\nonumber
=& \frac{1}{n}\sum_{j=1}^n\left(\mathbf{I}_3+\frac{1}{\K^2}\nabla_\bx\nabla_\bx^\top\right) e^{i\K\theta_j^\top(\bx-\by)}
\\
\nonumber
=& \left(\mathbf{I}_3+\frac{1}{\K^2}\nabla_\bx\nabla_\bx^\top\right)\left(\frac{1}{n}\sum_{j=1}^n e^{i\K\theta_j^\top(\bx-\by)}\right)
\\
\approx
&
-\frac{4\pi}{\K\epsilon_0}\Im m\big\{\mathbf{\Gamma}(\bx,\by)\big\}.\label{theta1}
\end{align}
Similarly, as $\ds\sum_{\ell=1}^2\left(\theta_j\times \theta_j^{\perp,\ell}\right)\left(\theta_j\times \theta_j^{\perp,\ell}\right)^\top= \left(\mathbf{I}_3-\theta_j\theta_j^\top\right)$ for all $j=1,\cdots,n$, we have
\begin{equation}
\frac{1}{n}\sum_{\ell=1}^2\sum_{j=1}^n  e^{i\K\theta_j^\top(\bx-\by)} \left(\theta_j\times \theta_j^{\perp,\ell}\right)\left(\theta_j\times \theta_j^{\perp,\ell}\right)^\top
\approx
-\frac{4\pi}{\K\epsilon_0}\Im m\big\{\mathbf{\Gamma}(\bx,\by)\big\},
\quad\bx,\by\in\RR^3,\, \bx\neq\by.\label{theta2}
\end{equation}

The following result holds.

\begin{thm}\label{multiResolution}
Let $\bz_S\in\RR^3$, $n\in\mathbb{N}$ be sufficiently large and $\rho\K\ll 1$. Then 
\begin{enumerate}
\item  for permittivity contrast only  ($\mu_0=\mu_1=\mu_2$) 
\begin{align}
\Itd^n(\bz_S)
\simeq \frac{\rho^3\K^4a_1^\epsilon a_2^\epsilon}{\epsilon_0^2}\Re e\Big\{\Im m\left\{\mathbf{\Gamma}(\bz_S,\bz_D)\right\} \mathbf{M}_{D}^\epsilon: \mathbf{M}_{S}^\epsilon\Im m\left\{\mathbf{\Gamma}(\bz_S,\bz_D)\right\} \Big\}
+ O(\rho^4),\label{TDmulti2}
\end{align}
\item for a permeability contrast only ($\epsilon_0=\epsilon_1=\epsilon_2$)
\begin{align}
\Itd^n(\bz_S) \simeq  \frac{\rho^3\K^4 a_1^\mu a_2^\mu}{\epsilon_0^2}\Re e\Big\{\Im m\left\{\mathbf{\Gamma}(\bz_S,\bz_D)\right\}\mathbf{M}_{D}^\mu:\mathbf{M}_{S}^\mu
\Im m\left\{\mathbf{\Gamma}(\bz_S,\bz_D)\right\}\Big\} +O(\rho^4).\label{TDmulti1}
\end{align}
\end{enumerate}
\end{thm}

\begin{proof}
By definition of the multi-measurement location indicator function and the asymptotic expansion \eqref{TDeps2} of $\Itd$ we have 
\begin{align*}
\Itd^n(\bz_S)
=& 
- \frac{\rho^3\K^3a_1^\epsilon a_2^\epsilon}{4\pi\epsilon_0n} \sum_{\ell=1}^2\sum_{j=1}^n
\Re e\Big\{
\Im m\left\{\mathbf{\Gamma}(\bz_S,\bz_D)\right\}\mathbf{M}^\epsilon_D
\overline{\mathbf{E}_0^{j,\ell}(\bz_D)}
\cdot \mathbf{M}_S^\epsilon
\mathbf{E}_0^{j,\ell}(\bz_S)\Big\} +O(\rho^4)
\\
=& 
- \frac{\rho^3\K^5a_1^\epsilon a_2^\epsilon}{4\pi\epsilon_0n} \sum_{\ell=1}^2 \sum_{j=1}^n
\Re e\Big\{
\Im m\left\{\mathbf{\Gamma}(\bz_S,\bz_D)\right\}\mathbf{M}^\epsilon_D
(\theta_j\times\theta_j^{\perp,\ell})\cdot
\\
&
\qquad\qquad\qquad\qquad\qquad\qquad\qquad\qquad\qquad
\mathbf{M}_S^\epsilon
(\theta_j\times\theta_j^{\perp,\ell})e^{i\K\theta_j^\top(\bz_S-\bz_D)}
\Big\} +O(\rho^4)
\end{align*}
for all search points $\bz_S\in\RR^3$ when the inclusion has only permittivity contrast. Since $\mathbf{A}\theta\cdot\theta =\mathbf{A}: \theta\theta^\top$ for any matrix $\mathbf{A}\in\RR^{3\times 3}$ and $\theta\in\mathbb{S}^2$, the approximation \eqref{theta2} yields
\begin{align*}
\Itd^n(\bz_S)
=& 
- \frac{\rho^3\K^5a_1^\epsilon a_2^\epsilon}{4\pi\epsilon_0n} \sum_{\ell=1}^2\sum_{j=1}^n\Re e\Big\{\Im m\left\{\mathbf{\Gamma}(\bz_S,\bz_D)\right\}\mathbf{M}^\epsilon_D :
\\
& \qquad\qquad\qquad\qquad\qquad\qquad 
\mathbf{M}_S^\epsilon
(\theta_j\times\theta_j^{\perp,\ell})(\theta_j\times\theta_j^{\perp,\ell})^\top e^{i\K\theta_j^\top(\bz_S-\bz_D)}
\Big\} +O(\rho^4)
\\
\simeq & \phantom{-} 
\frac{\rho^3\K^4a_1^\epsilon a_2^\epsilon}{\epsilon_0^2}
\Re e\Big\{
\Im m\left\{\mathbf{\Gamma}(\bz_S,\bz_D)\right\}\mathbf{M}^\epsilon_D
:\mathbf{M}_S^\epsilon \Im m\left\{\mathbf{\Gamma}(\bz_S,\bz_D)\right\}
\Big\} +O(\rho^4).
\end{align*}
Therefore, the first assertion is proved. 

In order to prove the second assertion  note that by virtue of expansion \eqref{TDmu2} for a permeable inclusion and by the fact that $\theta_j\perp\theta_j^{\perp,\ell}$, the function $\Itd^n(\bz_S)$ takes on the form
\begin{align*}
\Itd^n(\bz_S)
=& 
- \frac{\rho^3\K a_1^\mu a_2^\mu}{4\pi\epsilon_0n}
\sum_{\ell=1}^2\sum_{j=1}^n
\Re e\left\{
\Im m\big\{\mathbf{\Gamma}(\bz_S,\bz_D)\big\} 
\mathbf{M}_D^\mu\overline{\nabla\times
\mathbf{E}_0^{j,\ell}(\bz_D)}\cdot \mathbf{M}_S^\mu\nabla\times \mathbf{E}_0^{j,\ell}(\bz_S)\right\}
\\
=& 
- \frac{\rho^3\K^5a_1^\mu a_2^\mu}{4\pi\epsilon_0n}
\sum_{\ell=1}^2\sum_{j=1}^n
\Re e\Big\{
\Im m\left\{\mathbf{\Gamma}(\bz_S,\bz_D)\right\}\mathbf{M}^\mu_D
\left(\theta_j\times(\theta_j\times\theta_j^{\perp,\ell})\right)
\cdot 
\\
&\qquad\qquad\qquad\qquad\qquad\qquad\qquad\qquad
\mathbf{M}_S^\mu
\left(\theta_j\times(\theta_j\times\theta_j^{\perp,\ell})\right)
e^{i\K\theta_j^\top(\bz_S-\bz_D)}
\Big\}+O(\rho^4)
\\
=& 
- \frac{\rho^3\K^5a_1^\mu a_2^\mu}{4\pi\epsilon_0n}\sum_{\ell=1}^2\sum_{j=1}^n
\Re e\Big\{
\Im m\left\{\mathbf{\Gamma}(\bz_S,\bz_D)\right\}\mathbf{M}^\mu_D
\theta_j^{\perp,\ell} \cdot  \mathbf{M}_S^\mu
\theta_j^{\perp,\ell} e^{i\K\theta_j^\top(\bz_S-\bz_D)}
\Big\}+O(\rho^4).
\end{align*}
The use of identity $\mathbf{A}\theta\cdot\theta =\mathbf{A}: \theta\theta^\top$ once again, together with approximation \eqref{theta1}, leads to  
\begin{align*}
\Itd^n(\bz_S)
=&  
- \frac{\rho^3\K^5a_1^\mu a_2^\mu}{4\pi\epsilon_0n} \sum_{\ell=1}^2\sum_{j=1}^n
\Re e\Big\{
\Im m\left\{\mathbf{\Gamma}(\bz_S,\bz_D)\right\}\mathbf{M}^\mu_D
: \mathbf{M}_S^\mu
\theta_j^{\perp,\ell} \left(\theta_j^{\perp,\ell}\right)^\top e^{i\K\theta_j^\top(\bz_S-\bz_D)}
\Big\}
\\
&+O(\rho^4)
\\
\simeq& 
\phantom{-.}
\frac{\rho^3\K^4a_1^\mu a_2^\mu}{\epsilon_0^2}
\Re e\Big\{
\Im m\left\{\mathbf{\Gamma}(\bz_S,\bz_D)\right\}\mathbf{M}^\mu_D
: \mathbf{M}_S^\mu\Im m\big\{\mathbf{\Gamma}(\bz_S,\bz_D)\big\}
\Big\}+O(\rho^4).
\end{align*}
This completes the prove.
\end{proof}

As an immediate consequence of  Theorem \ref{multiResolution} and Lemma \ref{PT}, the following result is evident.

\begin{cor}\label{cor}
Let $\bz_S\in\RR^3$, $D=\rho B_D +\bz_D$ be a spherical inclusion with only permittivity contrast or permeability contrast, $n\in\mathbb{N}$ be sufficiently large and $\rho\K\ll 1$. Then  
\begin{align*}
\Itd^n(\bz_S) \simeq & \rho^3\K^4{C}_\gamma\left\|\Im m\Big\{\mathbf{\Gamma}(\bz_S,\bz_D)\Big\}\right\|^2 +O(\rho^4),
\end{align*}
where  
\begin{eqnarray*}
C_\gamma:= \frac{9(\gamma_0-\gamma_1)(\gamma_0-\gamma_2)}{\epsilon_0^2(2\gamma_0+\gamma_1)(2\gamma_0+\gamma_2)}|B_D|\,|B_S|
\end{eqnarray*}
and $\gamma$ denotes $\epsilon$ (resp. $\mu$) for an inclusion with only permittivity (resp. permeability) contrast.
\end{cor}

Theorem \ref{multiResolution} and Corollary \ref{cor} substantiate that $\Itd^n(\bz_S)\propto \|\Im m\left\{\mathbf{\Gamma}(\bz_S,\bz_D)\right\}\|^2$. In view of the decay properties of imaginary part of the fundamental solution  it is evident that $\Itd^n$ has a peak sharper than that of $\Itd$ when $\bz_S\to\bz_D$. Moreover, $\Itd^n$ admits its most pronounced positive value if the contrasts of true and trial inclusions have same signs as the constant $C_\gamma$ is positive in this case.
In order to further elaborate our findings a numerical experiment for the detection of a spherical inclusion, centered at origin and having radius $\rho=0.01$ with only permittivity contrast, is presented in Figure \ref{fig:5}. 
We choose 
$$
\theta_j=(\sin(\phi_m)\cos(\psi_n),\sin(\phi_m)\sin(\psi_n),\cos(\phi_m))^\top,
$$  
where 
$$\phi_m=\frac{(m-1) \pi}{M},
\quad 
\psi_n=\frac{2(n-1)\pi}{N},
\quad 
\theta_j^{\perp,1}=\frac{\theta_j\times \mathcal{R}\theta_j}{|\theta_j\times \mathcal{R}\theta_j|}, \quad 
\theta_j^{\perp,2}=\theta_j^{\perp,1}\times\theta_j.
$$
Here $\mathcal{R}$ is a rotation defined by 
$$
\mathcal{R}=
\begin{pmatrix}1&0&0
\\ 0&0&1\\0&-1&0
\end{pmatrix}.
$$   
We choose $\mu_0=1=\mu_1$, $\epsilon_0=1$, $\epsilon_1=2$, $\K=8\pi$, and plot $\bz_S\mapsto\Itd^n(\bz_S)$ for all $\bz_S\in [-1,1]\times [-1,1]\times\{0\}$ using $101\times 101$ sampling points. The numerical results in Figure \ref{fig:5} clearly show a significant increase in the amplitude of the peak of $\bz_S\mapsto\Itd^n(\bz_S)$ with an increase in the number of incident fields. Moreover, a stabilization effect in terms of the ripples around the main peak is visible. 
\begin{figure}[!ht]
\includegraphics[width=0.43\textwidth]{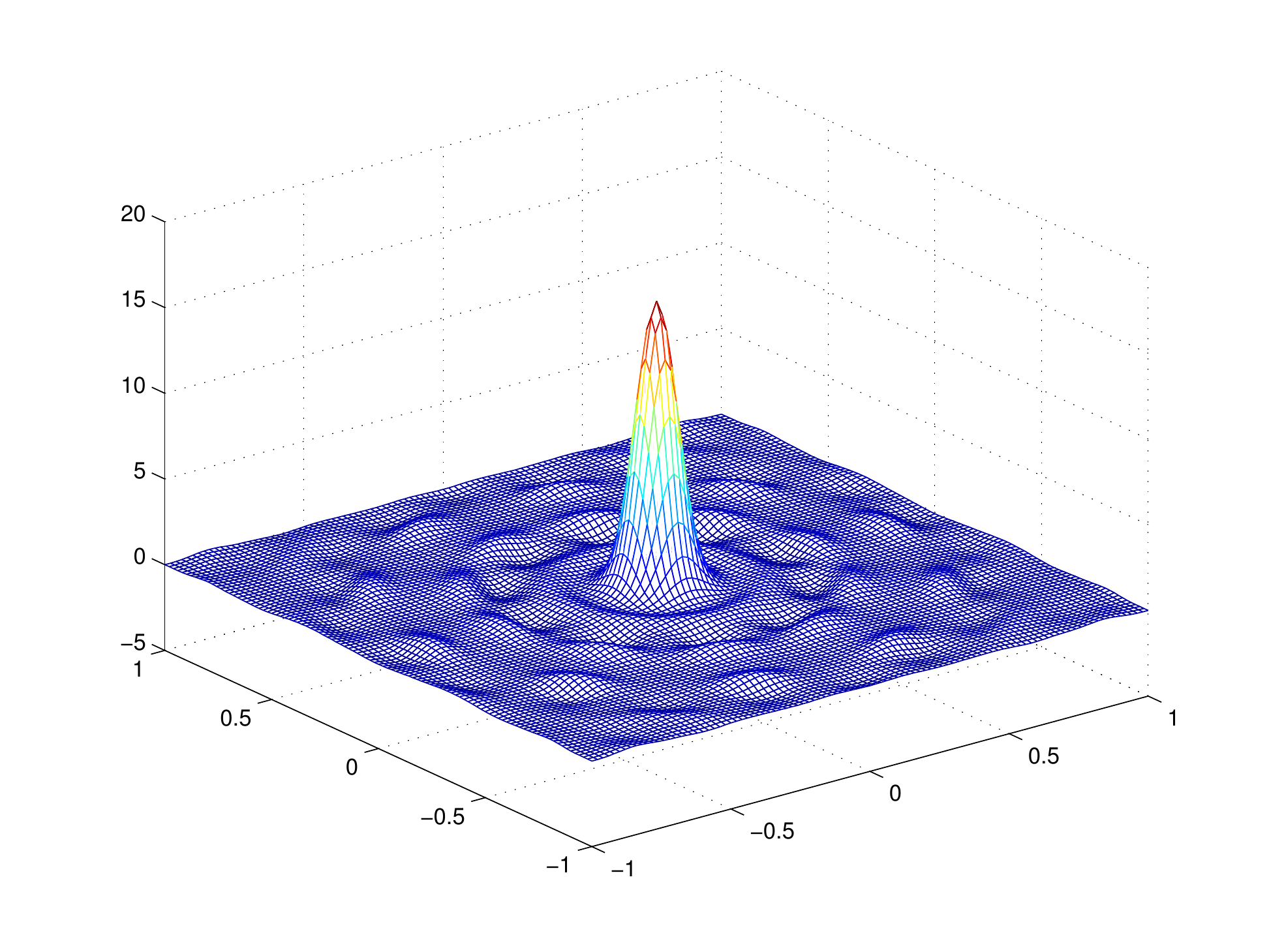}
\includegraphics[width=0.43\textwidth]{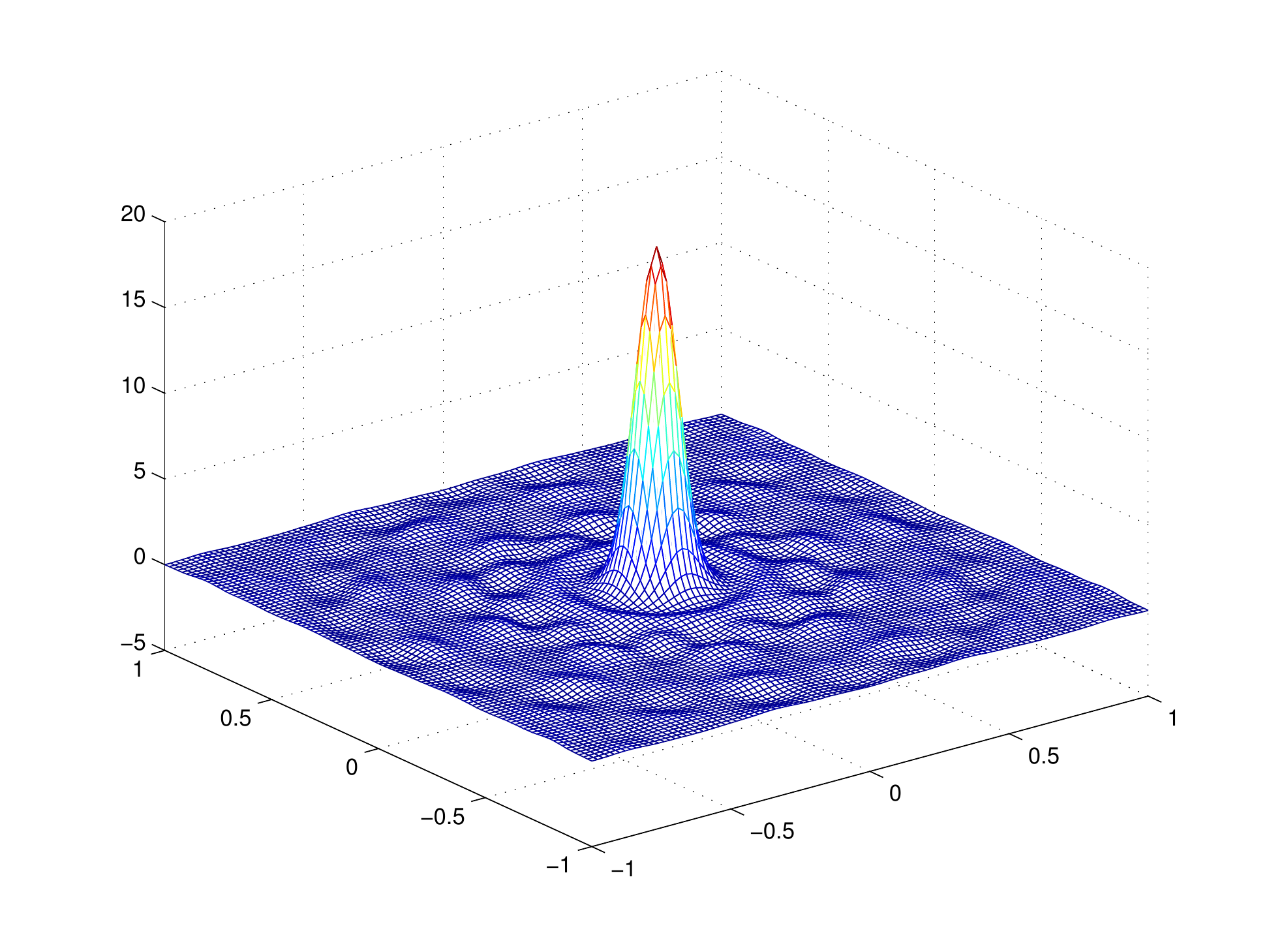}
\includegraphics[width=0.48\textwidth]{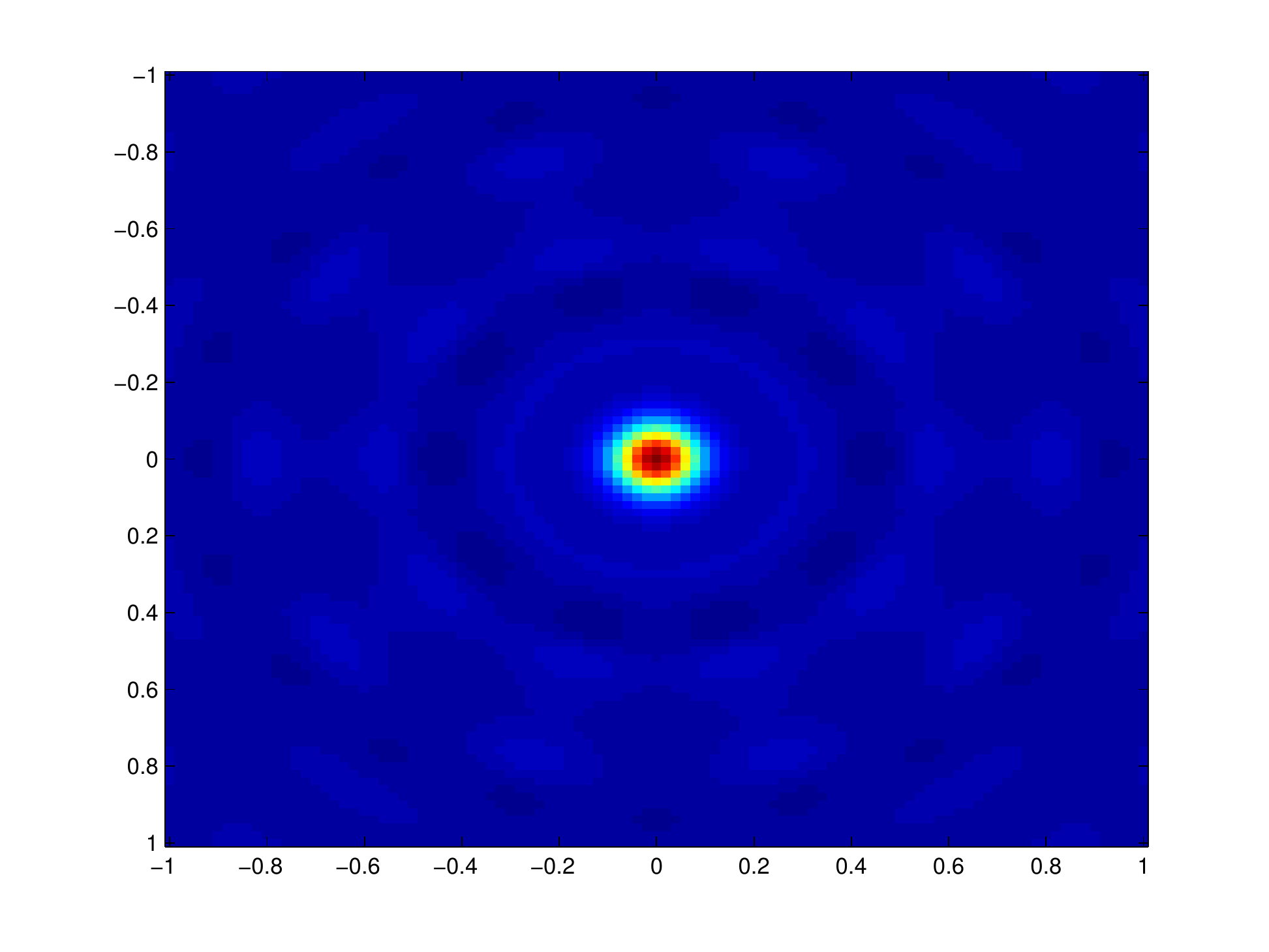}
\includegraphics[width=0.48\textwidth]{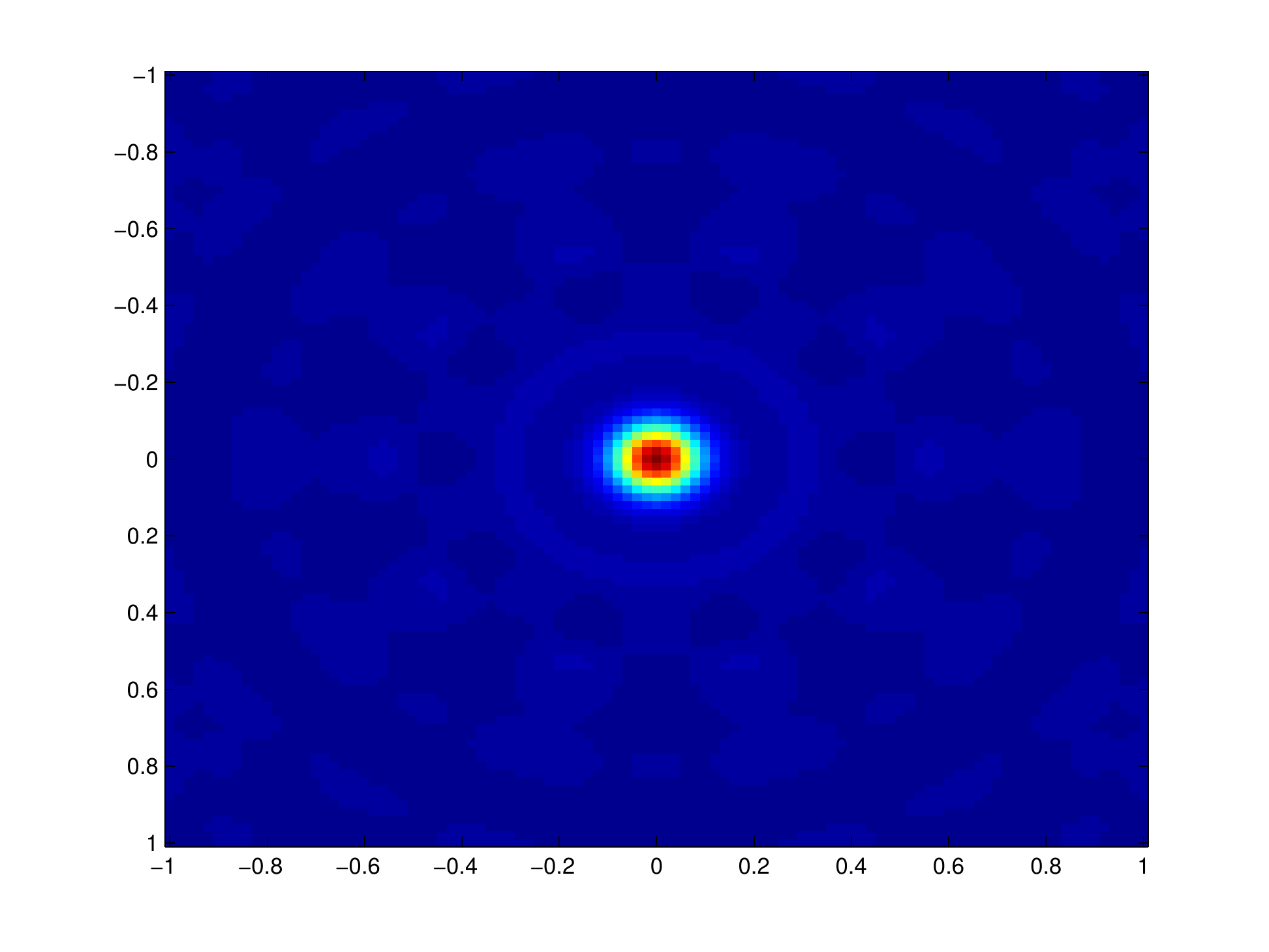}
\caption{Detection of a spherical inclusion with permittivity contrast. Left: $50$ incident fields ($N=M=5$). Right: $200$ incident fields ($N=M=10$). Top: $\bz_S\mapsto\Itd^n(\bz_S)$. Bottom: Focal spot.}\label{fig:5}
\end{figure}

In the rest of this article, the stability of the multi-measurement location indicator function $\Itd^n(\bz_S)$ with respect to medium and measurement noises is investigated.

\section{Statistical stability  with respect to measurement noise}\label{measNoise}

The aim in this section is to prove the statistical stability of the function $\Itd^n(\bz_S)$ with respect to additive measurement noise. The simplest model of measurement noise is taken into account. Precisely, it is assumed that the measurements of the far field amplitude are corrupted by a mean-zero circular Gaussian noise $\etab:\mathbb{S}^2\to\mathbb{C}^3$, that is, 
$$
\mathbf{E}_\rho^\infty(\hbx):= \widetilde{\mathbf{E}^{\infty}_\rho}(\hbx)+\etab(\hbx),\qquad \hbx\in\mathbb{S}^2,
$$
where $\mathbf{E}_\rho^\infty$ represents the corrupted far field data and $\widetilde{\mathbf{E}^{\infty}_\rho}(\hbx)$ indicates true data without noise corruption.  Let $\sigmae^2$ be the noise covariance of $\etab$ such that for all $\hby,\hby'\in\mathbb{S}^2$
\begin{eqnarray}
\mathbb{E}\left[\etab(\hby)\overline{\etab(\hby')}^\top\right]
=\sigmae^2\delta_\hby(\hby')\mathbf{I}_3
\quad\text{and}\quad
\mathbb{E}\left[\etab^j(\hby)\overline{\etab^{j'}(\hby')}^\top\right]
=\sigmae^2\delta_{jj'}\delta_\hby(\hby')\mathbf{I}_3,\label{EEcor}
\end{eqnarray}
where $j$ and $j'$ indicate the $j$-th and $j'$-th measurements and $\delta_{jj'}$ is Kronecker's delta function. Here $\mathbb{E}$ denotes the expectation with respect to the statistics of noise. By \eqref{EEcor} following assumptions are actually made. 
\begin{itemize}
\item[i.] $\etab(\hby)$ and $\etab(\hby')$ are uncorrelated for all $\hby,\hby'\in\mathbb{S}^2$ whenever $\hby\neq\hby'$.
\item[ii.] $\Im m\{\etab(\hby)\}$ and $\Re e\{\etab(\hby)\}$  are uncorrelated for all $\hby\in\mathbb{S}^2$.
\item[iii.] The components of $\etab$ are uncorrelated.
\item[iv.] $\etab^j$ and $\etab^{j'}$ are uncorrelated but have common noise covariance $\sigmae^2$ for all $j,j'\in\{1,\cdots,n\}$ whenever $j\neq j'$.
\end{itemize}

It is interesting to note that, in view of Theorem \ref{thmTD} and definition \eqref{MultiFunc}, the noise affects $\Itd^n$ through Herglotz wave. In fact, by linearity of the Herglotz operator 
\begin{eqnarray}
\mathcal{H}_E[\mathbf{E}^\infty_\rho]=\mathcal{H}_E[\widetilde{\mathbf{E}^{\infty}_\rho}]+ \mathcal{H}_E[\etab],
\label{HeNoise}
\end{eqnarray}
where the first term on the RHS, $\mathcal{H}_E[\widetilde{\mathbf{E}^{\infty}_\rho}]$, is independent of noise and renders the principle peak of $\Itd^n$, which is discussed earlier in Section \ref{sensitivity}. On the other hand, the second term on the RHS, $\mathcal{H}_E[\etab]$, is a circular Gaussian random process with mean-zero by linearity and by definition of $\etab$. Furthermore, since by construction  $\etab\in L^2_T(\mathbb{S}^2)$,  for all $\bz,\bz'\in\RR^3$ 
\begin{align}
\mathbb{E}\Big[\mathcal{H}_E[\etab](\bz)
\overline{\mathcal{H}_E[\etab](\bz')}^\top\Big]
=&\mathbb{E}\left[\iint_{\mathbb{S}^2\times\mathbb{S}^2}
\etab(\hbx)\overline{\etab(\hby)^\top} e^{i\K\hbx^\top\bz} e^{-i\K\hby^\top\bz'} ds(\hbx)ds(\hby)\right]
\nonumber
\\
=&\mathbb{E}\left[\iint_{\mathbb{S}^2\times\mathbb{S}^2}
\hbx\times(\etab(\hbx)\times\hbx)\overline{\etab(\hby)^\top} e^{i\K\hbx^\top\bz} e^{-i\K\hby^\top\bz'} ds(\hbx)ds(\hby)\right].
\nonumber
\end{align}
Then, by virtue of identities \eqref{relation2} and \eqref{exp1},  and subsequently using assumptions in \eqref{EEcor} on additive noise $\etab$, one gets for all $\bz,\bz'\in\RR^3$
\begin{align}
\mathbb{E}\Big[\mathcal{H}_E[\etab](\bz)
\overline{\mathcal{H}_E[\etab](\bz')}^\top\Big]
=&\mathbb{E}\left[\iint_{\mathbb{S}^2\times\mathbb{S}^2}
(\mathbf{I}_3-\hbx\hbx^\top)\etab(\hbx)\overline{\etab(\hby)^\top} e^{i\K\hbx^\top\bz} e^{-i\K\hby^\top\bz'} ds(\hbx)ds(\hby)\right]
\nonumber
\\
=&\sigmae^2\left(\mathbf{I}_3+\frac{1}{\K^2}\nabla_\bz\nabla_\bz^\top\right) \int_{\mathbb{S}^2}e^{i\K\hbx^\top(\bz-\bz')}ds(\hbx)  
\nonumber
\\
=&
-\frac{4\pi\sigmae^2}{\K\epsilon_0}\Im m\big\{\mathbf{\Gamma}(\bz,\bz')\big\}.\label{Cov1}
\end{align}
The covariance of the Herglotz field $\mathcal{H}_E[\etab]$ shows that it precipitates a speckle pattern, that is, a random cloud of hot spots which have typical diameters of the order of the operating wavelength and  amplitudes of the order of $\sigmae /\sqrt{\K}$.

In the sequel, let us again stick to the cases of inclusions with only permittivity or permeability contrast. In order to understand the effect of noise on location detection capabilities of topological sensitivity functional, the covariance of indicator function $\Itd^n$ and the signal-to-noise ratios are derived in the following subsections.

\subsection{Permittivity contrast}\label{ss:meas:1}

Let us first entertain the case of an inclusion with permittivity contrast only when permeability parameters of both $D$ and $D_S$ match $\mu_0$, and $\Itd$ is provided by \eqref{TDeps}.  Then  the covariance of the corrupted location indicator function $\Itd^n$ is given by 
\begin{align*}
{\rm Cov}
&\big(\Itd^n(\bz),\Itd^n(\bz')\big)
\\
=&
\frac{\K^4(a_2^\epsilon)^2}{16\pi^2n^2}\sum_{\ell,\ell'=1}^2\sum_{j,j'=1}^n
\mathbb{E}\Big[
\ds\Re e\left\{
\overline{\mathcal{H}_E
[\etab^{j,\ell}](\bz)}
\cdot \mathbf{M}_S^\epsilon
\mathbf{E}_0^{j,\ell}(\bz)\right\}
\Re e\left\{
\overline{\mathcal{H}_E
[\etab^{j',\ell'}](\bz')}
\cdot \mathbf{M}_S^\epsilon
\mathbf{E}_0^{j',\ell'}(\bz')\right\}
\Big] 
\\
=
&
\frac{\K^4(a_2^\epsilon)^2}{32\pi^2n^2}\sum_{\ell=1}^2\sum_{j=1}^n
\ds\Re e\left\{\mathbf{M}_S^\epsilon
\mathbf{E}_0^{j,\ell}(\bz)\cdot 
\mathbb{E}\Big[
\overline{\mathcal{H}_E
[\etab^{j,\ell}](\bz)}
\Big(\mathcal{H}_E [\etab^{j,\ell}](\bz')\Big)^\top\Big]
\mathbf{M}_S^\epsilon
\overline{\mathbf{E}_0^{j,\ell}(\bz')}\right\},
\end{align*}
for all $\bz,\bz'\in\RR^3$, where the fact that $\mathcal{H}_E[\etab^{j,\ell}]$ and $\mathcal{H}_E[\etab^{j',\ell'}]$ are uncorrelated  for $j\neq j'$ and $\ell\neq \ell'$  is used.  By virtue of statistics \eqref{Cov1}  the covariance of $\Itd^n$  turns out to be 
\begin{align*}
{\rm Cov}
&\big(\Itd^n(\bz),\Itd^n(\bz')\big)
=
-\frac{\sigmae^2\K^3(a_2^\epsilon)^2}{8\pi n^2\epsilon_0}\sum_{\ell=1}^2\sum_{j=1}^n
\ds\Re e\left\{\mathbf{M}_S^\epsilon
\mathbf{E}_0^{j,\ell}(\bz)\cdot 
\Im m\big\{\mathbf{\Gamma}(\bz,\bz')\big\}
\mathbf{M}_S^\epsilon
\overline{\mathbf{E}_0^{j,\ell}(\bz')}\right\}.
\end{align*}
The above expression can be further simplified 
by substituting expression \eqref{Ej} for $\mathbf{E}^{j,\ell}_0$ and using approximation \eqref{theta2} as
\begin{align}
\nonumber
{\rm Cov}\big(&\Itd^n(\bz),\Itd^n(\bz')\big)
\\
=
&
-\frac{\sigmae^2\K^5(a_2^\epsilon)^2}{8\pi n^2\epsilon_0}
\sum_{\ell=1}^2\sum_{j=1}^n
\ds\Re e\left\{\mathbf{M}_S^\epsilon(\theta_j\times\theta_j^{\perp,\ell})\cdot 
\Im m\big\{\mathbf{\Gamma}(\bz,\bz')\big\}
\mathbf{M}_S^\epsilon
(\theta_j\times\theta_j^{\perp,\ell}) e^{i\K\theta^\top(\bz-\bz')}\right\}
\nonumber
\\
=
&
-\frac{\sigmae^2\K^5(a_2^\epsilon)^2}{8\pi n\epsilon_0}
\ds\Re e\Bigg\{\mathbf{M}_S^\epsilon\frac{1}{n}
\sum_{\ell=1}^2\sum_{j=1}^n(\theta_j\times\theta_j^{\perp,\ell})
(\theta_j\times\theta_j^{\perp,\ell})^\top e^{i\K
\theta^\top(\bz-\bz')}: 
\Im m\big\{\mathbf{\Gamma}(\bz,\bz')\big\}
\mathbf{M}_S^\epsilon
\Bigg\} 
\nonumber
\\
\approx
&\phantom{-.}
\frac{\sigmae^2\K^4(a_2^\epsilon)^2}{2n\epsilon_0^{2}}
\ds\Re e\Big\{\mathbf{M}_S^\epsilon \Im m\big\{\mathbf{\Gamma}(\bz,\bz')\big\}: 
\Im m\big\{\mathbf{\Gamma}(\bz,\bz')\big\}
\mathbf{M}_S^\epsilon
\Big\},\qquad\forall\bz,\bz'\in\RR^3.\label{cov:meas:1}
\end{align}

The covariance \eqref{cov:meas:1} of the perturbation in $\Itd^n$ appears to be exactly of the form \eqref{TDmulti2} of indicator function in the absence of noise. However, it is modulated by the factor $\sigmae^2/n$. Thus the hot spots in the speckle field  generated by the propagation of noise through $\Itd^n$ have shapes identical to that of actual peak procured by $\bz_S\mapsto\Itd^n(\bz_S)$ subject to accurate measurements. Nevertheless, the main spike of $\bz_S\mapsto\Itd^n(\bz_S)$ is not altered thanks to relation \eqref{HeNoise}. Further, the dependence of typical perturbation size on $\sqrt{n}$ suggests that the topological derivative based location indicator functions are more stable when multiple measurements are available at hand and increasing $n$ further enhances the stability. 

Now, assume for an instance that inclusions $D$ and $D_S$ are spherical. Then  it can be readily verified by virtue of Lemma \ref{PT} that
\begin{align}
{\rm Cov}\big(&\Itd^n(\bz),\Itd^n(\bz')\big)
\approx
\ds\sigmae^2\K^4|B_S|^2\frac{9(\epsilon_0-\epsilon_2)^2}{2n\epsilon_0^2(2\epsilon_{0}+\epsilon_2)^2}
\left\|\Im m\big\{\mathbf{\Gamma}(\bz,\bz')\big\}\right\|^2.\label{Cov2}
\end{align}
It follows immediately from \eqref{Cov2} that the variance of  $\Itd^n$  can be approximated by 
\begin{align}
{\rm Var}\big(\Itd^n(\bz)\big)
\approx &
\ds\sigmae^2\K^4|B_S|^2\frac{9(\epsilon_0-\epsilon_2)^2}{2n\epsilon_0^2(2\epsilon_{0}+\epsilon_2)^2}
\big\|\Im m\big\{\mathbf{\Gamma}(\bz,\bz)\big\}
\big\|^2
\nonumber
\\
=&
\ds \sigmae^2\K^6|B_S|^2\frac{3(\epsilon_0-\epsilon_2)^2}{8n\pi^2(2\epsilon_{0}+\epsilon_2)^2}.
\label{var1} 
\end{align}
Consequently  the signal-to-noise ratio (SNR), defined by 
$$
{\rm SNR}= \frac{\mathbb{E}\left[\Itd^n(\bz_D)\right]}{\sqrt{{\rm Var}\left[\Itd^n(\bz_D)\right]}},
$$
can be approximated by virtue of \eqref{var1} and Corollary \ref{cor} as
\begin{eqnarray}
{\rm SNR}\approx 
\frac{\sqrt{6}}{2\pi(2\epsilon_0+\epsilon_1)} \,\rho^3|B_D|\,|\epsilon_0-\epsilon_1|\,\K^3\,\sqrt{n}\, \frac{1}{\sigmae}.\label{SNR}
\end{eqnarray}
Clearly  SNR depends directly on contrast, volume of the inclusion, the operating frequency and number of plane incident fields, and is inversely proportional to the noise covariance. Specifically, the dependence of SNR on $n$ shows that a better resolution can be achieved in the presence of measurement noise by increasing the number of incident fields. 


\subsection{Permeability contrast}
Let us now consider the case of an inclusion with permeability contrast only, $\mu_2=\mu_0$ and $\Itd$ is given by relation \eqref{TDmu}. In order to evaluate the covariance of the corrupted indicator function  the statistics 
$$
\mathbb{E}\Big[\nabla\times\mathcal{H}_E[\etab](\bz)\overline{\nabla\times\mathcal{H}_E[\etab](\bz')}^\top\Big]
$$ 
is required. Towards this end, note that
\begin{align}
\mathbb{E}\Big[\nabla\times\mathcal{H}_E[\etab](\bz)
&
\overline{\nabla\times\mathcal{H}_E[\etab](\bz')}^\top\Big]
\nonumber
\\
=&
\K^2\mathbb{E}\left[\iint_{\mathbb{S}^2\times\mathbb{S}^2}
\left[\hbx\times\etab(\hbx)e^{i\K\hbx^\top\bz}\right]\left[\overline{\hby\times\etab(\hby)}e^{-i\K\hby^\top\bz'}\right]^\top   ds(\hbx)ds(\hby)\right] 
\nonumber
\\
=&\K^2\mathbb{E}\iint_{\mathbb{S}^2\times\mathbb{S}^2}
\left(\mathbf{I}_3\times\hbx\right)^\top\left[\etab(\hbx)\overline{\etab(\hby)}^\top\right]\left(\mathbf{I}_3\times\hby\right)e^{i\K\hbx^\top\bz} e^{-i\K\hby^\top\bz'} ds(\hbx)ds(\hby) 
\nonumber
\\
=&
\sigmae^2\K^2\int_{\mathbb{S}^2}
\hbx\times\left(\mathbf{I}_3\times\hbx\right) e^{i\K\hbx^\top(\bz-\bz')} ds(\hbx),
\nonumber
\end{align}
where the fact that $\mathbf{A}^\top(\mathbf{p}\times\mathbf{q})=(\mathbf{A}\times \mathbf{p})^\top\mathbf{q}$ for all $\mathbf{A}\in\RR^{3\times 3}$, $\mathbf{p},\mathbf{q}\in\RR^3$ is used together with assumption \eqref{EEcor}. Finally, identity \eqref{relation2} is invoked to get 
\begin{align}
\mathbb{E}\left[\nabla\times\mathcal{H}_E[\etab](\bz)
\overline{\nabla\times\mathcal{H}_E[\etab](\bz')}^\top\right]
=&\sigmae^2\K^2\int_{\mathbb{S}^2}
(\mathbf{I}_3-\hbx\hbx^\top)e^{i\K\hbx^\top(\bz-\bz')} ds(\hbx) 
\nonumber
\\
=&\sigmae^2\K^2\left(\mathbf{I}_3+\frac{1}{\K^2}\nabla_\bz\nabla_\bz^\top\right) \int_{\mathbb{S}^2}e^{i\K\hbx^\top(\bz-\bz')}ds(\hbx)   
\nonumber
\\
=&
-\frac{4\pi\sigmae^2\K}{\epsilon_0}\Im m\big\{\mathbf{\Gamma}(\bz,\bz')\big\}.\label{Cov3}
\end{align}

Let us now evaluate the covariance of $\Itd^n$. By definition, for the case of only permeability contrast,
\begin{align*}
{\rm Cov}
\big(\Itd^n(\bz),\Itd^n(\bz')\big)
=&
\frac{(a_2^\mu)^2}{16\pi^2n^2}\sum_{\ell,\ell'=1}^2\sum_{j,j'=1}^n
\mathbb{E}\Big[
\ds\Re e\left\{
\overline{\nabla\times\mathcal{H}_E
[\etab^{j,\ell}](\bz)}
\cdot \mathbf{M}_S^\mu
\nabla\times\mathbf{E}_0^{j,\ell}(\bz)\right\}
\\
&\qquad\qquad\qquad\qquad
\Re e\Big\{
\overline{\nabla\times\mathcal{H}_E
[\etab^{j',\ell'}](\bz')}
\cdot \mathbf{M}_S^\mu
\nabla\times\mathbf{E}_0^{j',\ell'}(\bz')\Big\}
\Big].
\end{align*} 
The above expression can be simplified by using the statistics \eqref{Cov3} together with the fact that $\nabla\times\mathcal{H}_E
[\etab^{j,\ell}](\bz)$ and $\nabla\times\mathcal{H}_E
[\etab^{j',\ell'}](\bz')$ are uncorrelated for all $j\neq j'$ and $\ell\neq \ell'$,  as
\begin{align*}
{\rm Cov}
\big(\Itd^n(\bz),&\Itd^n(\bz')\big)
\\=
&
\frac{(a_2^\mu)^2}{32\pi^2n^2}\sum_{\ell=1}^2\sum_{j=1}^n
\ds\Re e\Big\{\mathbf{M}_S^\mu
\nabla\times\mathbf{E}_0^{j,\ell}(\bz)\cdot 
\\
&\qquad\qquad\qquad\quad
\mathbb{E}\Big[
\overline{\nabla\times\mathcal{H}_E
[\etab^{j,\ell}](\bz)}
\Big(\nabla\times\mathcal{H}_E [\etab^{j,\ell}](\bz')\Big)^\top\Big]
\mathbf{M}_S^\mu
\overline{\nabla\times\mathbf{E}_0^{j,\ell}(\bz')}\Big\} 
\\
=
&
-\frac{\sigmae^2\K(a_2^\mu)^2}{8\pi n^2\epsilon_0}\sum_{\ell=1}^2\sum_{j=1}^n
\ds\Re e\left\{\mathbf{M}_S^\mu
\nabla\times\mathbf{E}_0^{j,\ell}(\bz)\cdot 
\Im m\big\{\mathbf{\Gamma}(\bz,\bz')\big\}
\mathbf{M}_S^\mu
\overline{\nabla\times\mathbf{E}_0^{j,\ell}(\bz')}\right\}.
\end{align*}
Finally, substituting the expression for $\nabla\times\mathbf{E}^{j,\ell}_0$ from \eqref{CurlE}, using the orthogonality of $\theta_j$ and $\theta_j^{\perp,\ell}$ for all $j=1,\cdots,n$, and invoking approximation \eqref{theta1}, one arrives at
\begin{align}
{\rm Cov}\big(\Itd^n(\bz),&\Itd^n(\bz')\big)
\nonumber
\\
=&
-\frac{\sigmae^2\K^5(a_2^\mu)^2}{8\pi n\epsilon_0}
\ds\Re e\Bigg\{\mathbf{M}_S^\mu\Bigg(\frac{1}{n}
\sum_{\ell=1}^2\sum_{j=1}^n\theta_j^{\perp,\ell}
(\theta_j^{\perp,\ell})^\top e^{i\K
\theta^\top(\bz-\bz')}\Bigg): 
\Im m\big\{\mathbf{\Gamma}(\bz,\bz')\big\}
\mathbf{M}_S^\mu
\Bigg\}
\nonumber
\\
\approx &\phantom{-.}
\frac{\sigmae^2\K^4(a_2^\mu)^2}{2n\epsilon_0^{2}}
\ds\Re e\Big\{\mathbf{M}_S^\mu \Im m\big\{\mathbf{\Gamma}(\bz,\bz')\big\}: 
\Im m\big\{\mathbf{\Gamma}(\bz,\bz')\big\}
\mathbf{M}_S^\mu
\Big\}.\label{CovPmeas}
\end{align}

As for the case of an inclusion with permittivity contrast, the statistics \eqref{CovPmeas} suggests that the speckle field generated by the noise $\etab$ is a random cloud of hot spots having the shapes identical to that of the principle peak modulated by the factor $\sigmae^2/n$. Moreover, the actual peak is not altered by the additive measurement noise. 
Further, when inclusions $D$ and $D_S$ are spherical, by Lemma \ref{PT} 
\begin{align}
{\rm Cov}\big(&\Itd^n(\bz),\Itd^n(\bz')\big)
\approx
\ds\sigmae^2\K^4|B_S|^2\frac{9(\mu_0-\mu_2)^2}{2n\epsilon_0^2(2\mu_{0}+\mu_2)^2}
\left\|\Im m\big\{\mathbf{\Gamma}(\bz,\bz')\big\}\right\|^2 \label{Cov4}
\end{align}
and consequently
\begin{align}
{\rm Var}\big(\Itd^n(\bz)\big)
\approx 
\ds \sigmae^2\K^6|B_S|^2\frac{3(\mu_0-\mu_2)^2}{8n\pi^2(2\mu_{0}+\mu_2)^2}.
\label{var2} 
\end{align}
Thus,  by virtue of \eqref{var2} and Corollary \ref{cor} 
\begin{eqnarray}
{\rm SNR}\approx 
\frac{\sqrt{6}}{2\pi(2\mu_0+\mu_1)}\, \rho^3|B_D|\,|\mu_0-\mu_1|\,\K^3\,\sqrt{n}\,\frac{1}{\sigmae}.\label{SNR2}
\end{eqnarray}
The expressions \eqref{Cov4}--\eqref{SNR2} for a permeable inclusion are very similar to \eqref{Cov2}--\eqref{SNR} obtained for the case of a dielectric inclusion. Therefore, the same conclusions can be drawn about the statistics of noise perturbation in $\Itd^n$ for an inclusion with permeability contrast as in the case of permittivity contrast already discussed in Section \ref{ss:meas:1}.


\section{Statistical stability with respect to medium noise}\label{medNoise}
Let us now investigate the stability of detection functional $\Itd^n$ with respect to medium noise. Assume for simplicity that only one of the permittivity or permeability fluctuates around the background value at a time. In this section  the homogeneous medium with parameters $\epsilon_0$ and $\mu_0$ in the absence of any inclusion is termed as the reference medium and the quantities related to the reference medium are distinguished by a superposed $0$. The inhomogeneous random medium without inclusion is termed as the background medium. 

\subsection{Random fluctuations in permeability}
\label{s:medNoise1}
Let the electromagnetic material loaded in $\RR^3$ be inhomogeneous such that its permeability randomly fluctuates around $\mu_0$. The fluctuating permeability is denoted by $\mu$ and is modeled as 
$$
\mu(\bx):=\mu_0\left(1+\eta(\bx)\right),\qquad \bx\in\RR^3,
$$
where $\eta:\RR^3\to\RR$ is the random fluctuation with typical size $\sigma_{\eta}$, which is small enough so that the Born approximation is valid. 
 
In order to fathom the role of medium noise in inclusion detection using $\Itd^n$  one needs to model the noisy measurements $\mathbf{E}^\infty_\rho$ and the back-propagator $\mathcal{H}_E[\mathbf{E}^\infty_\rho]$. Let us first consider $\mathbf{E}^\infty_\rho$  defined by
\begin{eqnarray}
\mathbf{E}_\rho(\bx)-\mathbf{E}_0^0(\bx)=\frac{e^{i\K|\bx|}}{|\bx|}\mathbf{E}^\infty_\rho(\hbx;\theta,\theta^\perp)+O\left(\frac{1}{|\bx|^2}\right), \quad\bx\in\RR^3,\theta\in\mathbb{S}^2.
\label{Einfty2}
\end{eqnarray}
Note that the reference incident field $\mathbf{E}_0^0$ is used since the differential measurements are with respect to the reference solution and the background medium is randomly fluctuating. This indicates that the far field scattering pattern is contaminated with \emph{clutter noise} induced by fluctuation $\eta$. In order to separate noise from the original signal in $\mathbf{E}^\infty_\rho$,  the quantity $\mathbf{E}_\rho(\bx)-\mathbf{E}_0^0(\bx)$ is expressed as the sum of $\mathbf{E}_\rho(\bx)-\mathbf{E}_0(\bx)$ and $\mathbf{E}_0(\bx)-\mathbf{E}_0^0(\bx)$  which can be approximated in Born regime. Indeed, under Born approximation,  
\begin{align}
\label{Born1}
\mathbf{E}_0(\bx) &=\mathbf{E}_0^0(\bx)-\frac{\K^2}{\epsilon_0}\int_{\RR^3} \mathbf{\Gamma}^0(\bx,\by)\eta(\by)\mathbf{E}^0_0(\by)d\by +o(\sigma_\eta).
\end{align}
Since the fundamental solution $\mathbf{\Gamma}^0$ can be expanded as
\begin{align*}
\mathbf{\Gamma}^0(\bx,\by)= -\frac{e^{i\K|\bx|}}{|\bx|}\left[\epsilon_0(\mathbf{I}_3-\hbx\hbx^\top)\frac{e^{-i\K\hbx^\top\by}}{4\pi}\right]+ O\left(\frac{1}{|\bx|^2}\right)
\end{align*}
for a large $\bx\in\RR^3$ and a fixed $\by\in\RR^3$, therefore \eqref{Born1} can be expanded as
\begin{align}
\mathbf{E}_0(\bx)-\mathbf{E}_0^0(\bx)
&= 
\ds\frac{e^{i\K|\bx|}}{|\bx|}\frac{\K^2}{4\pi}\int_{\RR^3} (\mathbf{I}_3-\hbx\hbx^\top)e^{-i\K\hbx^\top\by} \mathbf{E}^0_0(\by)\eta(\by)d\by +o(\sigma_\eta)+ O\left(\frac{1}{|\bx|^2}\right)
\nonumber
\\
&= 
\ds\frac{e^{i\K|\bx|}}{|\bx|}\frac{\K^2}{4\pi}\int_{\RR^3} \Big[\mathbf{I}_3+\frac{1}{\K^2}\nabla_\by\nabla_\by^\top\Big]e^{-i\K\hbx^\top\by} \mathbf{E}^0_0(\by)\eta(\by)d\by +o(\sigma_\eta)+ O\Big(\frac{1}{|\bx|^2}\Big).
\label{Born2}
\end{align}
Moreover, under Born approximation, 
\begin{align}
\label{Born3}
\mathbf{E}_\rho(\bx)- \mathbf{E}_0(\bx) 
=&
\Big(\mathbf{E}_\rho^0(\bx)- \mathbf{E}_0^0(\bx)\Big)-\frac{\K^2}{\epsilon_0}\int_{\RR^3} \mathbf{\Gamma}^0(\bx,\by)\eta(\by)\left(\mathbf{E}_\rho^0(\by)- \mathbf{E}_0^0(\by)\right)d\by +o(\sigma_\eta).
\end{align}
The first term on the RHS of \eqref{Born3} is exactly the reference scattered field which is linked to the reference far field scattering amplitude by relation \eqref{Einfty}. Further, Theorem \ref{Asymp} indicates that the second term on the RHS is of order $o(\sigma_\eta\rho^3)$. Therefore,  
\begin{align}
\label{Born4}
\mathbf{E}_\rho(\bx)- \mathbf{E}_0(\bx) 
=&
\frac{e^{i\K|\bx|}}{|\bx|}\mathbf{E}^{\infty,0}_\rho(\hbx;\theta,\theta^\perp)+o(\sigma_\eta)+o(\sigma_\eta\rho^3)+O\left(\frac{1}{|\bx|^2}\right).
\end{align}
Combining \eqref{Einfty2}, \eqref{Born2} and \eqref{Born4}  the inhomogeneous far field scattering amplitude can be written as  
\begin{align*}
\mathbf{E}^\infty_\rho(\hbx,\theta,\theta^\perp)
=& 
\mathbf{E}^{\infty,0}_\rho(\hbx,\theta,\theta^\perp)+\frac{\K^2}{4\pi}\int_{\RR^3} \Big[\mathbf{I}_3+\frac{1}{\K^2}\nabla_\by\nabla_\by^\top\Big]e^{-i\K\hbx^\top\by} \mathbf{E}^0_0(\by)\eta(\by)d\by
\nonumber
\\
&
+o(\sigma_\eta)+o(\sigma_\eta\rho^3)+ O\Big(\frac{1}{|\bx|^2}\Big),
\end{align*}
and consequently the Herglotz field due to $\mathbf{E}^\infty_\rho$ admits the expansion 
\begin{align*}
\mathcal{H}_E[\mathbf{E}^\infty_\rho(\cdot,\theta,\theta^\perp)](\bz)
=&
\mathcal{H}_E[\mathbf{E}^{\infty,0}_\rho(\cdot,\theta,\theta^\perp)](\bz)
\\
&+\frac{\K^2}{4\pi}\int_{\RR^3}\left(\Big[\mathbf{I}_3+\frac{1}{\K^2}\nabla_\by\nabla_\by^\top\Big]\int_{\mathbb{S}^2} e^{i\K\hbx^\top(\bz-\by)} ds(\hbx)\right)\mathbf{E}^0_0(\by)\eta(\by)d \by
\\
&
+o(\sigma_\eta)+o(\sigma_\eta\rho^3)+ O\Big(\frac{1}{|\bx|^2}\Big),
\end{align*}
for all $\bz\in\RR^3$. Finally, by identity \eqref{exp1},
\begin{align}
\mathcal{H}_E[\mathbf{E}^\infty_\rho(\cdot,\theta,\theta^\perp)](\bz)
=&
\mathcal{H}_E[\mathbf{E}^{\infty,0}_\rho(\cdot,\theta,\theta^\perp)](\bz)-\frac{\K}{\epsilon_0}\int_{\RR^3}\Im m\big\{\mathbf{\Gamma}(\bz,\by)\big\}\mathbf{E}^0_0(\by)\eta(\by)d \by
\nonumber
\\
&
+o(\sigma_\eta)+o(\sigma_\eta\rho^3)+ O\Big(\frac{1}{|\bx|^2}\Big).\label{H2}
\end{align}

The first term on the RHS of \eqref{H2} is the Herglotz field in the reference medium and is independent of random fluctuations $\eta$. It produces the principle peak of $\Itd^n$, that is, without medium noise. The second term corresponds to the error due to clutter and generates a speckle pattern corrupting the location indicator function. Therefore, the expression \eqref{H2} together with Theorem \ref{thmTD} substantiates that the functional $\Itd^n$ in the case of medium noise is the sum of the reference detection functional (studied in Section \ref{sensitivity}), and the clutter error up to leading order of approximation.

Before further discussion, let us introduce the clutter error operator $\mathcal{H}_E^\eta$ for all incident fields $\mathbf{E}$ and $\bx\in\RR^3$ by 
\begin{eqnarray}
\mathcal{H}_E^\eta[\mathbf{E}](\bx):=-\frac{\K}{\epsilon_0}\int_{\RR^3}\Im m\big\{\mathbf{\Gamma}(\bx,\by)\big\}\mathbf{E}(\by)\eta(\by)d \by.
\label{HClutter}
\end{eqnarray}
For ease of notation  the clutter error associated to the reference incident field $\mathbf{E}^{0,j,\ell}_0$ (and thereby associated to $\mathbf{E}^\infty_\rho(\cdot,\theta_j,\theta_j^{\perp,\ell})$) will be denoted by $\mathcal{H}_E^{\eta,j,\ell}$ instead of $\mathcal{H}_E^\eta[\mathbf{E}^{0,j,\ell}_0]$.
In the rest of this subsection   the speckle pattern generated by $\mathcal{H}_E^{\eta,j,\ell}$ will be analyzed for inclusions with only permittivity or permeability contrasts.  

\subsubsection{Speckle field analysis for permittivity contrast}\label{ss:speckle11}

In order to understand the speckle field generated by $\mathcal{H}_E^{\eta,j,\ell}$ for $\ell=1,2$ and $j=1,\cdots, n$  when $\mu_0=\mu_1=\mu_2$  let us calculate the covariance of the  perturbation in function $\Itd^n$ due to clutter noise. Towards this end, first the quantity $\frac{1}{n}\sum_{\ell=1}^2\sum_{j=1}^n 
\overline{\mathcal{H}_E^{\eta, j,\ell}
(\bz)} \cdot\mathbf{M}_{S}^\epsilon\mathbf{E}^{0,j,\ell}_0(\bz)$ is estimated for all $\bz\in\RR^3$.
Using \eqref{HClutter}, \eqref{Ej} and \eqref{theta2} respectively one can see that  
\begin{align*}
\frac{1}{n}\sum_{\ell=1}^2\sum_{j=1}^n 
&\overline{\mathcal{H}_E^{\eta, j,\ell}
(\bz)} \cdot\mathbf{M}_{S}^\epsilon\mathbf{E}^{0,j,\ell}_0(\bz)
\\
= &
-\frac{\K}{n\epsilon_0}\sum_{\ell=1}^2\sum_{j=1}^n\int_{\RR^3} \eta(\by)\Im m\left\{\mathbf{\Gamma}^0(\by,\bz)\right\}:\mathbf{M}_{S}^\epsilon\mathbf{E}^{0,j,\ell}_0 (\bz) \left[\overline{\mathbf{E}^{0,j,\ell}_0(\by)}\right]^\top d\by 
\\
= &
-\frac{\K^3}{\epsilon_0}\int_{\RR^3}\eta(\by)\Im m\left\{\mathbf{\Gamma}^0(\by,\bz)\right\}:\mathbf{M}_{S}^\epsilon
\Bigg[
\frac{1}{n}\sum_{\ell=1}^2\sum_{j=1}^n(\theta_j\times\theta_j^{\perp,\ell})\Big[\theta_j\times\theta_j^{\perp,\ell}\Big]^\top
e^{i\K\theta_j^\top(\bz-\by)}
\Bigg] d\by 
\\
\simeq &
\phantom{-}\frac{4\pi\K^2}{\epsilon_0^2}\int_{\RR^3} \eta(\by)\mathcal{Q}_\eta[\mathbf{M}_{S}^\epsilon](\by,\bz)d\by,
\end{align*}
where $\mathcal{Q}_\eta$ is a real valued function defined by
$$
\mathcal{Q}_\eta[\mathbf{A}](\bx,\by):=\Im m\big\{\mathbf{\Gamma}^0(\bx,\by)\big\}:\mathbf{A} \Im m\big\{\mathbf{\Gamma}^0(\bx,\by)\big\}, \quad\forall \mathbf{A}\in\RR^{3\times 3},\,\,\bx,\by\in\RR^3.
$$
Therefore, the covariance  of the speckle field  can be approximated by 
\begin{align*}
\nonumber
{\rm Cov}
&\big(\Itd^n(\bz),\Itd^n(\bz')\big)
\\
& =\frac{\K^4(a_2^\epsilon)^2}{16\pi^2n^2}\sum_{\ell,\ell'=1}^2\sum_{j,j'=1}^n
\mathbb{E}\Big[
\ds\Re e\left\{
\overline{\mathcal{H}_E^{\eta,j,\ell}(\bz)}
\cdot \mathbf{M}_S^\epsilon
\mathbf{E}_0^{0,j,\ell}(\bz)\right\}
\Re e\left\{
\overline{\mathcal{H}_E^{\eta, j',\ell'}(\bz')}
\cdot \mathbf{M}_S^\epsilon
\mathbf{E}_0^{0,j',\ell'}(\bz')\right\}
\Big]   
\nonumber
\\
&\simeq  \frac{\K^8(a_2^\epsilon)^2}{\epsilon_0^4} \iint_{\RR^3\times\RR^3} C_\eta(\bx,\by)\mathcal{Q}_\eta[\mathbf{M}_{S}^\epsilon](\bx,\bz)\mathcal{Q}_\eta[\mathbf{M}_{S}^\epsilon](\by,\bz') d\bx d\by, 
\end{align*}
for all $\bz,\bz'\in\RR^3$, where $C_\eta(\bx,\by)=\mathbb{E}\left[\eta(\bx)\eta(\by)\right]$  is  the two-point correlation function of the fluctuations in permittivity.  The kernel $\bz\mapsto \mathcal{Q}_\eta[\mathbf{M}_S^\epsilon](\bz,\bz')$ for a fixed $\bz'\in\RR^3$ is maximal when $\bz\to\bz'$ and the focal spot of its peak is of the order of half the operating wavelength. 

For simplicity assume that the inclusion $D$ is spherical and note that $\mathcal{Q}_\eta[\mathbf{I}_3](\bx,\by)= \|\Im m\big\{\mathbf{\Gamma}(\bx,\by)\big\}\|^2$. Then, by Lemma \ref{PT},  
\begin{align}
&\frac{1}{n}\sum_{\ell=1}^2\sum_{j=1}^n 
\overline{\mathcal{H}_E^{\eta, j,\ell}
(\bz)} \cdot\mathbf{M}_{S}^\epsilon\mathbf{E}^{0,j,\ell}_0(\bz)\simeq
\frac{12\pi\epsilon_2\K^2}{\epsilon_0^2(2\epsilon_0+\epsilon_2)}|B_S|\,\int_{\RR^3} \eta(\by)\mathcal{Q}_\eta[\mathbf{I}_3](\by,\bz)d\by,\label{Approx2}
\\
&{\rm Cov}\big(\Itd^n(\bz),\Itd^n(\bz')\big)
\simeq b_\epsilon^2 \K^8 \iint_{\RR^3\times\RR^3} C_\eta(\bx,\by)
\mathcal{Q}_\eta[\mathbf{I}_3](\bx,\bz)
\mathcal{Q}_\eta[\mathbf{I}_3](\by,\bz') d\bx d\by,
\label{ApproxCov}
\end{align}
where 
\begin{eqnarray*}
b_\epsilon= \frac{3(\epsilon_0-\epsilon_2)}{\epsilon_0^2(2\epsilon_0+\epsilon_2)}|B_S|.
\end{eqnarray*}
The expression \eqref{Approx2} elucidates that the speckle field in the image is essentially the medium noise smoothed by an integral kernel of the form $\|\Im m\{{\mathbf{\Gamma}}^0\}\|^2$. Similarly, \eqref{ApproxCov} shows that the correlation structure of the speckle field is essentially that of the medium noise smoothed by the same kernel. Since the typical width of $\Im m\{{\mathbf{\Gamma}}^0\}$ is about half the wavelength, the correlation length of the speckle field is roughly the maximum between the correlation length of the medium noise and the wavelength, that is, of the same order as the main peak centered at location $\bz_S\approx \bz_D$. Thus, there is no way to distinguish the main peak from the hot spots of the speckle field based on their shapes. Only the height of the main peak can allow it to be visible out of the speckle field.  Unlike measurement noise case discussed in the previous section, the factor $\sqrt{n}$ has disappeared. Consequently, no further stability enhancement is possible. Therefore, $\Itd^n$ is moderately stable with respect to medium noise. Moreover, the main peak of $\Itd^n$ is affected by the clutter, unlike in the measurement noise case. Thus, $\Itd^n$ is more efficient with respect to measurement noise than medium noise.

\subsubsection{Speckle field analysis for permeability contrast}\label{ss:speckle12}

Let us now consider the case of an inclusion with permeability contrast only when $\epsilon_0=\epsilon_1=\epsilon_2$. For the purpose of evaluating   covariance of the speckle field in this case  we proceed in the similar fashion as before. Note that 
\begin{align*}
\frac{1}{n}\sum_{\ell=1}^2\sum_{j=1}^n 
&\overline{\nabla\times \mathcal{H}_E^{\eta, j,\ell}
(\bz)} \cdot\mathbf{M}_{S}^\mu\nabla\times\mathbf{E}^{0,j,\ell}_0(\bz)
\\
= &
-\frac{\K}{n\epsilon_0}\sum_{\ell=1}^2\sum_{j=1}^n\int_{\RR^3} \eta(\by)\nabla_\bz\times\Im m\left\{\mathbf{\Gamma}^0(\by,\bz)\right\}\overline{\mathbf{E}^{0,j,\ell}_0(\by)}d\by\cdot
\mathbf{M}_{S}^\mu\nabla\times\mathbf{E}^{0,j,\ell}_0(\bz) 
\\
= &
-\frac{\K}{n\epsilon_0}\sum_{\ell=1}^2\sum_{j=1}^n\int_{\RR^3} \eta(\by)\nabla_\bz\times\Im m\left\{\mathbf{\Gamma}^0(\by,\bz)\right\}:\mathbf{M}_{S}^\mu\nabla_\bz\times\left(\mathbf{E}^{0,j,\ell}_0 (\bz) \left[\overline{\mathbf{E}^{0,j,\ell}_0(\by)}\right]^\top\right) d\by,
\end{align*}
for all $\bz\in\RR^3$. The above quantity can be approximated by using \eqref{Ej} and \eqref{theta2} as
\begin{align*}
\frac{1}{n}\sum_{\ell=1}^2\sum_{j=1}^n 
\overline{\nabla\times \mathcal{H}_E^{\eta, j,\ell}
(\bz)} \cdot\mathbf{M}_{S}^\mu\nabla\times\mathbf{E}^{0,j,\ell}_0(\bz)
=&-\frac{\K^3}{\epsilon_0}\int_{\RR^3}\eta(\by)\nabla_\bz\times\Im m\left\{\mathbf{\Gamma}^0(\by,\bz)\right\}:\mathbf{M}_{S}^\mu
\nabla_\bz\times
\\
& 
\Bigg[
\frac{1}{n}\sum_{\ell=1}^2\sum_{j=1}^n(\theta_j\times\theta_j^{\perp,\ell})\Big[\theta_j\times\theta_j^{\perp,\ell}\Big]^\top
e^{i\K\theta_j^\top(\bz-\by)}
\Bigg] d\by 
\\
\simeq &
\phantom{-}\frac{4\pi\K^2}{\epsilon_0^2}\int_{\RR^3} \eta(\by)\widetilde{\mathcal{Q}}_\eta[\mathbf{M}_{S}^\mu](\by,\bz)d\by,
\end{align*}
where the real valued function $\widetilde{\mathcal{Q}}_\eta$ is defined by
$$
\widetilde{\mathcal{Q}}_\eta[\mathbf{A}](\bx,\by):=
\nabla_\by\times\Im m\Big\{\widehat{\mathbf{\Gamma}}^0(\bx,\by)\Big\}:\mathbf{A} \nabla_\by\times\Im m\Big\{\widehat{\mathbf{\Gamma}}^0(\bx,\by)\Big\}, \quad\mathbf{A}\in\RR^{3\times 3},\,\bx,\by\in\RR^3.
$$
Therefore, the speckle field covariance in permeable contrast case can be approximated by 
\begin{align}
{\rm Cov} \big(\Itd^n(\bz),\Itd^n(\bz')\big)
\nonumber
=&\frac{(a_2^\mu)^2}{16\pi^2n^2}\sum_{\ell,\ell'=1}^2\sum_{j,j'=1}^n
\mathbb{E}\Big[
\ds\Re e\left\{
\overline{\nabla\times\mathcal{H}_E^{\eta,j,\ell}(\bz)}
\cdot \mathbf{M}_S^\mu
\nabla\times\mathbf{E}_0^{0,j,\ell}(\bz)\right\}
\nonumber
\\
&\qquad\qquad\qquad\qquad
\Re e\left\{
\overline{\nabla\times\mathcal{H}_E^{\eta, j',\ell'}(\bz')}
\cdot \mathbf{M}_S^\mu
\nabla\times\mathbf{E}_0^{0,j',\ell'}(\bz')\right\}
\Big] 
\nonumber
\\
\simeq &
\frac{\K^4(a_2^\mu)^2}{\epsilon_0^4} \iint_{\RR^3\times\RR^3} C_\eta(\bx,\by)\widetilde{\mathcal{Q}}_\eta[\mathbf{M}_{S}^\mu](\bx,\bz)\widetilde{\mathcal{Q}}_\eta[\mathbf{M}_{S}^\mu](\by,\bz') d\bx d\by,
\label{Cov6}
\end{align}
for all $\bz,\bz'\in\RR^3$. As in the previous case, the function $\bz\mapsto\widetilde{\mathcal{Q}}_\eta[\mathbf{M}_S^\mu](\bz,\bz')$ for a fixed $\bz'\in\RR^3$ is maximal when $\bz\to\bz'$ and the focal spot of its peak is of the order of half the operating wavelength.

Since  $\widetilde{\mathcal{Q}}_\eta[\mathbf{I}_3](\bx,\by)= \|\nabla\times\Im m\big\{\mathbf{\Gamma}(\bx,\by)\big\}\|^2$  it is evident that for a spherical inclusion
\begin{align*}
&\frac{1}{n}\sum_{\ell=1}^2\sum_{j=1}^n 
\overline{\nabla\times\mathcal{H}_E^{\eta, j,\ell}
(\bz)}\cdot\mathbf{M}_{S}^\mu\nabla\times\mathbf{E}^{0,j,\ell}_0(\bz)
\simeq
\frac{12\pi\mu_2\K^2|B_S|}{\epsilon_0^2(2\mu_0+\mu_2)}\,\int_{\RR^3} \eta(\by)\widetilde{\mathcal{Q}}_\eta[\mathbf{I}_3](\by,\bz)d\by,
\\
&{\rm Cov}\big(\Itd^n(\bz),\Itd^n(\bz')\big)
\simeq b_\mu^2 \K^4 \iint_{\RR^3\times\RR^3} C_\eta(\bx,\by)
\widetilde{\mathcal{Q}}_\eta[\mathbf{I}_3](\bx,\bz)
\widetilde{\mathcal{Q}}_\eta[\mathbf{I}_3](\by,\bz') d\bx d\by,
\end{align*}
where 
\begin{eqnarray*}
b_\mu= \frac{3(\mu_0-\mu_2)}{\epsilon_0^2(2\mu_0+\mu_2)}|B_S|.
\end{eqnarray*}
The conclusions drawn in Section \ref{ss:speckle11} still hold in this case, that is, $\Itd^n$ is moderately stable.

\subsection{Random fluctuations in permittivity}\label{s:medNoise2}

The aim in this section is to investigate the stability of the detection function $\Itd^n$ with respect to medium noise when the permittivity of the material loaded in $\RR^3$ is fluctuating.  Assume that the fluctuating background permittivity is given by 
\begin{eqnarray*}
{\epsilon^{-1}(\bx)}:= {\epsilon_0^{-1}}[1+\varphi(\bx)],
\end{eqnarray*} 
where the random fluctuation $\varphi:\RR^3\to\RR$ is small so that the Born approximation is suitable.  As a result,  
\begin{align*}
\mathbf{E}_0(\bx)-\mathbf{E}_0^0(\bx)=&-\frac{1}{\epsilon_0}\int_{\RR^3} \mathbf{\Gamma}^0(\bx,\by)\nabla\times\left(\varphi(\by)\nabla\times\mathbf{E}^0_0(\by)\right)d\by +o(\sigma_\varphi),
\\
\mathbf{E}_\rho(\bx)- \mathbf{E}_0(\bx) 
=&
\Big[\mathbf{E}_\rho^0(\bx)- \mathbf{E}_0^0(\bx)\Big]
-\frac{1}{\epsilon_0}\int_{\RR^3} \mathbf{\Gamma}^0(\bx,\by)\nabla\times\left(\varphi(\by)\nabla\times\Big[\mathbf{E}_\rho^0(\by)- \mathbf{E}_0^0(\by)\Big]\right)d\by 
\nonumber
\\
&+o(\sigma_\eta),
\end{align*} 
where $\sigma_\varphi$ is the amplitude of the fluctuation $\varphi$. Expressing $\mathbf{E}_\rho(\bx)- \mathbf{E}_0^0(\bx)$ as the sum of $\mathbf{E}_\rho(\bx)- \mathbf{E}_0(\bx)$ and $\mathbf{E}_0(\bx)- \mathbf{E}_0^0(\bx)$  and defining the inhomogeneous far field scattering amplitude $\mathbf{E}^\infty_\rho$ as in \eqref{Einfty2}  one gets   
\begin{align}
\mathbf{E}^\infty_\rho(\hbx,\theta,\theta^\perp)
=& 
\mathbf{E}^{\infty,0}_\rho(\hbx,\theta,\theta^\perp)+\frac{1}{4\pi}\int_{\RR^3} \Big[\mathbf{I}_3+\frac{1}{\K^2}\nabla_\by\nabla_\by^\top\Big]e^{-i\K\hbx^\top\by} \nabla\times\left(\varphi(\by)\nabla\times\mathbf{E}^0_0(\by)\right)d\by
\nonumber
\\
&
+o(\sigma_\eta)+o(\sigma_\eta\rho^3)+ O\Big(\frac{1}{|\bx|^2}\Big).
\label{Einfty4}
\end{align}

Now following the analysis in Section \ref{s:medNoise1}, it is easy to note that the Herglotz field, generated by the inhomogeneous far field scattering amplitude in \eqref{Einfty4}, consists of two terms: first one leading to the detection of true location of $D$ whereas the second one generates a speckle pattern due to clutter noise. The clutter error operator $\mathcal{H}_E^\varphi$ is now defined for all incident fields $\mathbf{E}$ and $\bx\in\RR^3$ by
\begin{eqnarray}
\mathcal{H}_E^\varphi[\mathbf{E}](\bx):=-\frac{1}{\K\epsilon_0}\int_{\RR^3}\Im m\big\{\mathbf{\Gamma}(\bx,\by)\big\}\nabla\times\left(\varphi(\by)\nabla\times\mathbf{E}(\by)\right)d \by.
\label{HClutter2}
\end{eqnarray}

\subsubsection{Speckle field analysis for permittivity contrast}

For permittivity contrast case the speckle field generated by $\Itd^n$  at $\bz\in\RR^3$ is given by 
\begin{align*}
\nonumber
\Lambda_\epsilon(\bz):=&
\frac{1}{n} \sum_{\ell=1}^2\sum_{j=1}^n 
\Re e\Big\{
\overline{\mathcal{H}_E^{\varphi,j,\ell}(\bz)}\cdot \mathbf{M}_S^\epsilon\mathbf{E}_0^{0,j,\ell}(\bz)
\Big\}
\\
\simeq &
-\frac{1}{n\K\epsilon_0}\sum_{j=1}^n 
\Re e\left\{
\int_{\RR^3} \Im m\left\{\mathbf{\Gamma}^0(\by,\bz)\right\}\left(\nabla\times\varphi(\by)\nabla
\times\overline{\mathbf{E}^{0,j,\ell}_0(\by)}\right)\cdot  
\mathbf{M}_S^\epsilon \mathbf{E}_0^{0,j,\ell}(\bz)d\by
\right\}.
\end{align*}
Since $\Im m\big\{\mathbf{\Gamma}^0(\by,\bz)\big\}$ is symmetric  therefore
\begin{align*}
\Lambda_\epsilon(\bz)\simeq &
-\frac{1}{n\K\epsilon_0}\sum_{\ell=1}^2\sum_{j=1}^n 
\Re e\left\{
\int_{\RR^3}\Im m\Big\{\mathbf{\Gamma}^0(\by,\bz)\Big\} \mathbf{M}_S^\epsilon\mathbf{E}_0^{0,j,\ell}(\bz)\cdot\left(\nabla\times\varphi(\by)\nabla
\times\overline{\mathbf{E}^{0,j,\ell}_0(\by)}\right)d\by
\right\}.
\end{align*}
Further assume that the fluctuation $\varphi$ is compactly supported in $\RR^3$, which is a valid assumption. Then the above expression simplifies to 
\begin{align*}
\Lambda_\epsilon(\bz) \simeq &
-\frac{1}{n\K\epsilon_0}\sum_{\ell=1}^2\sum_{j=1}^n 
\Re e\left\{
\int_{\RR^3} \nabla_\by\times\Im m\Big\{\mathbf{\Gamma}^0(\by,\bz)\Big\} \mathbf{M}_S^\epsilon \mathbf{E}_0^{0,j,\ell}(\bz)\cdot\varphi(\by)\nabla
\times\overline{\mathbf{E}^{0,j,\ell}_0(\by)}d\by
\right\}.
\end{align*}
Straight forward calculations and the approximation \eqref{theta2} show that 
\begin{align*}
\Lambda_\epsilon(\bz)
\simeq
 \frac{4\pi}{\epsilon_0^2}\Re e\left\{ \int_{\RR^3}\varphi(\by)\mathcal{Q}_\varphi[\mathbf{M}_S^\epsilon](\by,\bz)d\by 
\right\},
\end{align*}
where the real valued operator $\mathcal{Q}_\varphi$  is defined by  
$$
\mathcal{Q}_\varphi[\mathbf{A}](\bx,\by):=\nabla_\bx\times\Im m\big\{\mathbf{\Gamma}^0(\bx,\by)\big\}\mathbf{A}:\nabla_\bx\times\Im m\big\{\mathbf{\Gamma}^0(\bx,\by)\big\},
\quad\forall \mathbf{A}\in\RR^{3\times 3},\,\bx,\by\in\RR^3.
$$
Consequently, the covariance of the speckle field turns out to be
\begin{align}
\nonumber
{\rm Cov}\Big(\Itd^n(\bz), \Itd^n(\bz')\Big)
=&\frac{\K^4(a_2^\epsilon)^2}{16\pi^2}{\rm Cov}\big(\Lambda_\epsilon(\bz),\Lambda_\epsilon(\bz')\big)
\\
\simeq &
\frac{\K^4(a_2^\epsilon)^2}{\epsilon_0^4}\iint_{\RR^3\times\RR^3} C_\varphi(\bx,\by)\mathcal{Q}_\varphi[\mathbf{M}_S^\epsilon](\bx,\bz)\mathcal{Q}_\varphi[\mathbf{M}_S^\epsilon](\by,\bz')d\bx d\by, \label{Cov7}
\end{align}
where $C_\varphi(\bx,\by):=\mathbb{E}[\varphi(\bx)\varphi(\by)]$ is the two point correlation of fluctuation $\varphi$. The expression \eqref{Cov7} is very similar to \eqref{Cov6} studied in Section \ref{ss:speckle12}. As already pointed out in Section \ref{ss:speckle12}, the speckle field is indeed the medium noise smoothed with an integral kernel whose width is of the order of wavelength. The speckle pattern and its covariance for the case of a permittivity contrast can be obtained using arguments similar to those in the previous section.

\subsubsection{Speckle field analysis for permeability contrast}

For a permeability contrast the speckle field generated by $\Itd^n$  at $\bz\in\RR^3$ is given by 
\begin{align*}
& \Lambda_\mu(\bz):=\frac{1}{n}\sum_{\ell=1}^2\sum_{j=1}^n 
\overline{\nabla\times\mathcal{H}_E^{\varphi,j,\ell}(\bz)}\cdot  \mathbf{M}_S^\mu\nabla\times\mathbf{E}_0^{0,j,\ell}(\bz) 
\\
&= 
\frac{-1}{n\K\epsilon_0}\sum_{\ell=1}^2\sum_{j=1}^n 
\Re e\left\{
\int_{\RR^3} \nabla_\bz\times\Im m\Big\{\mathbf{\Gamma}^0(\by,\bz)\Big\}\nabla\times\varphi(\by)\nabla
\times\overline{\mathbf{E}^{0,j,\ell}_0(\by)}\cdot  \mathbf{M}_S^\mu\nabla\times \mathbf{E}_0^{0,j,\ell}(\bz)d\by
\right\}.
\end{align*}
Since $\varphi$ is assumed to have compact support, the above expression simplifies to
\begin{align*}
\Lambda_\mu(\bz) \simeq &
-\frac{1}{n\K\epsilon_0}\sum_{\ell=1}^2\sum_{j=1}^n 
\Re e\Big\{
\int_{\RR^3} \nabla_\by\times \nabla_\by\times\Im m\Big\{\mathbf{\Gamma}^0(\by,\bz)\Big\}\mathbf{M}_S^\mu \nabla\times \mathbf{E}_0^{0,j,\ell}(\bz)
\\
& \qquad\qquad\qquad\qquad\qquad\qquad\qquad\qquad\qquad\qquad\qquad\qquad\qquad
\cdot\varphi(\by)\nabla
\times\overline{\mathbf{E}^{0,j,\ell}_0(\by)} d\by \Big\} 
\\
= &
-\frac{\K}{n\epsilon_0}\sum_{\ell=1}^2\sum_{j=1}^n 
\Re e\left\{
\int_{\RR^3} \Im m\Big\{\mathbf{\Gamma}^0(\by,\bz)\Big\}\mathbf{M}_S^\mu\cdot \varphi(\by)\nabla \times\overline{\mathbf{E}^{0,j,\ell}_0(\by)}\left(\nabla\times \mathbf{E}_0^{0,j,\ell}(\bz)\right)^\top d\by
\right\}.
\end{align*} 
Finally, by invoking approximation \eqref{theta1} one arrives at
\begin{align*}
\Lambda_\mu(\bz)\simeq & \frac{4\pi\K^4}{\epsilon_0^2}
\int_{\RR^3}\varphi(\by)\widetilde{\mathcal{Q}}_\varphi
[\mathbf{M}^\mu_S](\by,\bz)d\by,
\end{align*} 
where $\widetilde{\mathcal{Q}}_\varphi$ is a real valued function defined by
$$
\widetilde{\mathcal{Q}}_\varphi[\mathbf{A}](\bx,\by)
=
\Im m\Big\{\mathbf{\Gamma}^0(\bx,\by)\Big\}\mathbf{A}: \Im m\Big\{\mathbf{\Gamma}^0(\bx,\by)\Big\},\quad \mathbf{A}\in\RR^{3\times 3},\,\bx,\by\in\RR^3.
$$
Consequently, 
\begin{align}
\nonumber
{\rm Cov}\Big(\Itd^n(\bz), \Itd^n(\bz')\Big)
=&\frac{\K^4(a_2^\mu)^2}{16\pi^2}{\rm Cov}\big(\Lambda_\mu(\bz),\Lambda_\mu(\bz')\big) 
\\
\simeq &
\frac{\K^8(a_2^\mu)^2}{\epsilon_0^4}\iint_{\RR^3\times\RR^3} C_\varphi(\bx,\by)\widetilde{\mathcal{Q}}_\varphi[\mathbf{M}_S^\mu](\bx,\bz)\widetilde{\mathcal{Q}}_\varphi[\mathbf{M}_S^\mu](\by,\bz').\label{Cov8}
\end{align}
Once again  it can be concluded from \eqref{Cov8} that the functional $\Itd^n$ is moderately stable with respect to medium noise. 

\section{Conclusions}\label{conc}

The topological sensitivity functions are presented for detecting an electromagnetic inclusion of vanishing characteristic size using single or multiple measurements of electric far field scattering amplitude. The proposed functions are very efficient in terms of resolution, stability and signal-to-noise ratio. The resolution limit of the order of half the operating wavelength is achieved. It is substantiated that the topological derivative based location indicator functions are very stable with respect to measurement noise, especially when multiple measurements are available at hand. Increasing the number $n$ of the measurements enhances the stability of the indicator functions with a rate of $O(\sqrt{n})$. It is also observed that the indicator functions are more efficient with respect to measurement noise than medium noise. Nevertheless, a moderate stability with respect to medium noise is guaranteed under Born approximation.  

\section*{Acknowledgment} 

The authors would like to thank Dr. Junqing Chen from Tsinghua University, Beijing for his invaluable suggestions and technical support.

\bigskip

\bibliographystyle{plain}

\end{document}